\documentclass[10pt, a4paper, reqno]{amsart}

\usepackage[T1]{url}
\usepackage[utf8]{inputenc}
\usepackage[left=1.15in,right=1.15in,top=1.15in,bottom=1.15in]{geometry}
\usepackage[UKenglish]{babel}
\usepackage[hidelinks]{hyperref} 
\usepackage{graphicx}
\usepackage{varioref, babel, fancyvrb, listings, amsmath,amsthm, amssymb, mathtools,bbm}
\usepackage{enumitem}
\usepackage{array}
\usepackage{booktabs}
\usepackage{mathabx}
\usepackage{cite}
\usepackage{times}
\usepackage{listings}
\usepackage[normalem]{ulem}
\usepackage[style=american]{csquotes}
\usepackage{tikz}
\usepackage{csquotes}

\usetikzlibrary{patterns}
\usetikzlibrary{arrows}
\usetikzlibrary{patterns}
\usetikzlibrary{shapes.misc}
\usetikzlibrary{arrows.meta} 
\tikzset{cross/.style={cross out, draw, 
         minimum size=2*(#1-\pgflinewidth), 
         inner sep=0pt, outer sep=0pt},
         cross/.default={3pt}}

\newcommand{\vertiii}[1]{{\left\vert\kern-0.25ex\left\vert\kern-0.25ex\left\vert #1 
    \right\vert\kern-0.25ex\right\vert\kern-0.25ex\right\vert}}

\def\R{\mathbb{R}}
\def\C{\mathbb{C}}
\def\N{\mathbb{N}}

\newcommand{\E}{\mathcal{E}}

\newcommand{\Mo}{\mathcal{M}_1}
\newcommand{\Mt}{\mathcal{M}_2}
\newcommand{\X}{\mathcal{X}}
\newcommand{\Y}{\mathcal{Y}}
\newcommand{\Z}{\mathcal{Z}}

\DeclareMathOperator*{\argmin}{arg\!\min}

\DeclareMathOperator*{\essup}{esssup}

\DeclareMathOperator*{\diam}{diam}

\theoremstyle{plain}
\newtheorem{theorem}{Theorem}[section]
\newtheorem{lemma}[theorem]{Lemma}
\newtheorem{proposition}[theorem]{Proposition}

\makeatletter
\@namedef{subjclassname@2020}{%
  \textup{2020} Mathematics Subject Classification}
\makeatother

\theoremstyle{definition}
\newtheorem{definition}[theorem]{Definition}
\newtheorem{assumption}{Assumption}

\newtheorem*{claim}{Claim}

\theoremstyle{remark}
\newtheorem{remark}[theorem]{Remark}

\DefineNamedColor{named}{Purple}{cmyk}{0.45,0.86,0,0}

\DefineNamedColor{named}{JungleGreen} {cmyk}{0.99,0,0.52,0}

\DefineNamedColor{named}{orange} {rgb}{1,0.55,0}

\DefineNamedColor{named}{ForestGreen} {cmyk}{0.76,0,0.76,0.45}

\title[Optimal set-valued decoders and their accuracy bounds]{On the existence of optimal set-valued decoders and their accuracy bounds for ill-posed inverse problems}  

\author{Nina M. Gottschling} 
\address{Department of Applied Mathematics and Theoretical Physics, University of Cambridge and now at Computing and Computational Sciences, Oak Ridge National Laboratory (ORNL), Oak Ridge, Tennessee.}
\email{gottschlinnm@ornl.gov}

\author{Paolo Campodonico} 
\address{Department of Applied Mathematics and Theoretical Physics, University of Cambridge}
\email{pc628@cam.ac.uk}

\author{Vegard Antun} 
\address{Department of Mathematics, University of Oslo}
\email{vegarant@math.uio.no}

\author{Anders C. Hansen} 
\address{Department of Applied Mathematics and Theoretical Physics, University of Cambridge}
\email{ach70@cam.ac.uk}

\date{}

\begin{document}

\newpage
\setcounter{page}{1}

\maketitle

\vspace{-5mm}	

\begin{abstract}
Ill-posed inverse problems occur everywhere in the sciences including medical imaging, radar, astronomy etc., yielding underdetermined or ill-posed linear (non-linear) reconstruction problems. There are now a myriad of techniques to design decoders/reconstruction-methods that can tackle such problems, ranging from optimization based approaches, such as compressed sensing, to data-driven techniques such as deep learning (DL), and variants in between the two techniques. The variety of methods begs for a unifying approach to determine the existence of optimal decoders and fundamental accuracy bounds, in order to facilitate a theoretical and empirical understanding of the performance of existing and future methods. Such a theory must allow for both single-valued and set-valued decoders, as underdetermined and ill-posed inverse problems typically have multiple solutions. Indeed, set-valued decoders arise due to non-uniqueness of minimizers in optimisation problems, such as in compressed sensing, and for DL based decoders in generative adversarial models, such as diffusion models and ensemble models. In this work we provide a framework for assessing the lowest possible reconstruction accuracy in terms of worst-case and average errors. The universal bounds only depend on the measurement model $F$, the model class $\mathcal{M}_1$ and the noise model $\mathcal{E}$. For linear $F$ these bounds depend on its kernel, and in the non-linear case the concept of kernel is generalized for undersampled and ill-posed settings. Additionally, we provide set-valued variational solutions that obtain the lowest possible reconstruction error.
\end{abstract}

\section{Introduction}

Finite dimensional inverse problems are ubiquitous in the computational sciences as they naturally appear in a plethora of applications. An incomplete list of examples includes most types of computational imaging \cite{adcock2021compressive, bouman2022foundations, barrett2013foundations}, matrix completion \cite{candes2012exact}, parametric PDEs/system identification \cite{adcock2022sparse}, phase retrieval \cite{shechtman2015phase, candes2013phaselift, fannjiang2020numerics}, quantized sampling \cite{4558487} etc. Most often, we can express these problems using the following general model:
\vspace{-2pt}
\begin{align}\label{eq:sampling1}
	\text{Given noisy measurements } y = F(x,e) \text{ of } x \in \Mo \text{ and }  e \in \E, \text{ recover } x.
\end{align}
Here $x$ represents the signal of interest, while $e$ represents the noise. The sets $\mathcal{M}_1$ and $\mathcal{E}$ describe the signals of interest and the potential noise, respectively. The function $F \colon \mathcal{M}_1\times \mathcal{E} \to \mathcal{M}_2^{\E}$ models the forward process, and is deliberately kept general to encompass many known models. Examples include
\begin{align}
	y &= G(x) + e,    &\text{(additive noise)} \\
	y &= G(x)\odot e, &\text{(multiplicative noise)} \\
	y &= G(x)\odot e_1 + e_2, &\text{(mixture of multiplicative and additive noise)}
\end{align}
where $G \colon \mathcal{M}_1 \to \mathcal{M}_2$ is a linear or non-linear forward map, and the dimension of the set $\mathcal{E}$ depends on the considered model in such a way that the point-wise multiplication $\odot$ is defined. 

\noindent In most instances of interest, the problem of recovering $x$ given $y$ is ill-posed or undersampled, unless further assumptions are made. An inverse problem is ill-posed, if it is not well-posed in the sense of Hadamard \cite{hadamard1902problemes,hadamard2003lectures}. In particular, it is ill-posed if either there does not exist a solution $x$ for every $y$, or the solution is not unique, or the solution does not depend continuously on $y$. We consider the case when the solution is not unique. This can be the case if the forward model $F$ is either poorly conditioned or dimensionality reducing. A standard assumption for making the problem \eqref{eq:sampling1} tractable is to assume that the noiseless forward model is injective when restricted to the model class $\mathcal{M}_1$. This assumption is a cornerstone for models with additive noise, as it allows for accurate reconstruction up to the noise level in both the linear \cite{traonmilin2018stable, FoucartRauhutCSbook, fundamental14} and non-linear \cite{keriven2018instance} setting.  Traditionally, the set $\mathcal{M}_1$ has been given a precise mathematical description and reconstruction methods have been designed for the given choice of $\mathcal{M}_1$. Examples of sets $\mathcal{M}_1$ include sparse vectors \cite{candes2008restricted}, union of subspaces \cite{blumensath2011sampling}, manifolds \cite{baraniuk2009random}, sparsity in a given basis \cite{adcock2017breaking} or frame \cite{poon2017structure}, matrices with low rank \cite{candes_oracle, NSP_low_rank}, etc. 

\noindent More recently, data-driven approaches \cite{Arr19, jin2017deep} have emerged as an alternative to many of these standard methods, often promising superior performance \cite{Bo-18}. Methods based on data typically do not specify the solution set $\mathcal{M}_1$, but aim to learn the reconstruction mapping $\Psi\colon \mathcal{M}_2\to \mathcal{M}_1$ given a finite number of training data $\mathcal{T}\subset \mathcal{M}_1$ and access to the forward model $F$ \cite{Arr19}. This approach has the advantage that the learned set $\mathcal{M}_1$ might provide a better approximation of the underlying data stemming from a real-world process, rather than a potentially simplistic abstract mathematical modelling of the set $\mathcal{M}_1$. However, the challenge with these methods is that the learned mapping $\Psi$ tends to produce accurate reconstructions of all elements in $\mathcal{T}$, regardless of whether the noiseless problem is injective on $\mathcal{M}_1$ or not. This has made these methods susceptible to both hallucinations \cite{fastmri20, SIREV_paper} and instabilities \cite{PNAS_paper}. 

\noindent The purpose of this work is to develop a mathematical framework for inverse problems which provides a notion of optimality when $F$ is not necessarily injective on $\mathcal{M}_1$, and to provide a variational expression for the optimal mappings in this setting. In summary this work answers the question posed by M. Burger and T. Roith in \textit{Learning in Image Reconstruction: A Cautionary Tale} \cite{burger2024learning},\newline

    "\textit{Can we ever quantify the reconstruction error [of ill-posed inverse problems]?}"\newline

\noindent Our main contributions are the following: 
\begin{enumerate}[label=(\Roman*)]
    \item\label{it:opt_map} \emph{We provide a framework which allows for the classification of the lowest achievable reconstruction accuracy for \eqref{eq:sampling1}.} Note that many such frameworks exist in the literature. For example, Gelfand widths in the noiseless linear setting \cite{pinkus2012n}, best $k$-term approximation \cite{CoDaDe-08}, and to obtain the worst-case bounds set-valued decoders on Banach spaces  have been considered in \cite{arestov1986optimal} and then extended to metric spaces in \cite{magaril1991optimal}, and when noise is included one has the notion of optimal learning \cite{binev2022optimal} and generalized instance optimality \cite{fundamental14}. See also \cite{micchelli1977survey} for a survey of early works on optimal recovery. 
    In many of these works, the underlying spaces are infinite or finite dimensional Banach spaces, and the reconstruction error is measured by the given norm. A contribution of this work is to consider the general setting of metric spaces, to allow for set-valued decoders which naturally arise as solutions of undersampled inverse problems, and to extend previous studies on worst-case scenarios to the average error scenario.  
    \item \emph{We derive explicit formulas for set-valued reconstruction mappings which achieve the lowest possible reconstruction error, both in a worst-case sense and in a probabilistic sense.}
    Finding and analyzing reconstruction mappings which achieve optimal or near-optimal recovery -- given some notion of optimality -- is an essential question in inverse problems, and many mappings exist \cite{fundamental14, binev2022optimal, plaskota1996noisy, traonmilin2018stable, FoucartRauhutCSbook}. In this work, we provide explicit formulas for possibly set-valued reconstruction mappings that achieve the lowest possible reconstruction error. Moreover, a contribution of this work is to prove that the optimal mappings in the average error scenario are measurable, compact valued and admit a measurable single-valued selector by utilising the Measurable Maximum Theorem, \cite[Thm.\ 18.19]{guide2006infinite} and Hausdorff distance. As such, this work provides a first and necessary step, before any numerical procedure can provide statistical approximations to these mappings. 
    \item \emph{We provide lower and upper bounds on the lowest achievable reconstruction accuracy described in }\ref{it:opt_map}. The bounds solely depend on the forward operator, signal and noise class of the inverse problem, and not on the method or decoder used to solve \eqref{eq:sampling1}. Hence, the performance of any -- possibly set-valued -- method can be evaluated with respect to the problem's intrinsic optimal accuracy. However, obtaining lower bounds come at a price: in contrast to compressed sensing results, which typically provide upper bounds (see \cite{FoucartRauhutCSbook, fundamental14}), fewer assumptions are required and for a non-zero lower bound the reconstruction error can not vanish. Specifically, the trade-off is as follows: lower bounds can be derived without assuming that the forward model $F$ is injective on $\Mo$, at the cost that the reconstruction error admits a non-zero limit. A key point in order to obtain lower bounds on reconstruction accuracy is to consider worst-case and average errors instead of point-wise approximation errors. Such worst-case error bounds can be related to the Chebyshev radius, as in \cite{binev2022optimal}, and analogous worst-case and average error bounds are established in \cite{plaskota1996noisy} in the case of normed spaces and single-valued decoders.
\end{enumerate}

\begin{remark}
	Extending previous work from normed spaces to metric spaces is of particular relevance in imaging applications. In fact, the $\ell^2$ or $\ell^1$ norm are not the only candidates considered for image quality assessment. A prominent example is the structural similarity index (SSI), firstly proposed to assess image quality in \cite{wang2004image}, does not constitute a norm but can rather be related to a metric \cite{brunet2011mathematical}. Note that the proposed framework encompasses all metrics, except trivial degenerate metrics.
\end{remark}

\begin{remark}[{\bf Set-valued decoders, computability and randomised algorithms}] Many of the decoders provided in theory will not be computable, as the phenomenon of generalised hardness of approximation \cite{comp_stable_NN22, opt_big, gazdag2022generalised, AIM, paradox22} -- within the Solvability Complexity Index (SCI) hierarchy \cite{Hansen_JAMS, Ben_Artzi2022, SCI, Colbrook_2019, Matt1} -- happen in many inverse problems. This includes a wide range of neural network decoders, as well as compressed sensing decoders. Typically, they can only be computed to a certain accuracy $\epsilon_0 > 0$, referred to as the approximation threshold, and it is the size of the approximation threshold that determines if the decoder can be used in practice.  However, as demonstrated in \cite{opt_big}, when the decoder is set-valued, randomised algorithms may actually help mitigating the non-computability.  This is only possible for set-valued decoders, as single-valued decoders that are non-computable will have no help from randomised algorithms \cite{LeeuwMooreShannonShapiro, opt_big}.  
\end{remark}

\section{Preliminaries}

\subsection{The forward map and the model class}
In the following we introduce the main assumptions and setup.
Let $\mathcal{X}$, $\mathcal{Y}$ and $\mathcal{Z}$ be non-empty sets. Furthermore, consider subsets  
$\mathcal{M}_1 \subset \mathcal{X}$ and $\mathcal{E}\subset \mathcal{Z}$, referred to as the \emph{model class} and the \emph{noise class}, respectively. As outlined in the introduction, we consider inverse problems with a given forward map $F \colon \mathcal{M}_1\times \mathcal{E} \to \mathcal{M}_2^{\mathcal{E}} \subset \mathcal{Y}$, with 
\[
\mathcal{M}_2^{\mathcal{E}} = F(\mathcal{M}_1\times\E) = \{y \in \mathcal{Y}: \exists (x,e) \in \mathcal{M}_1\times \mathcal{E}\text{ s.t.\ } y=F(x,e) \} 
\] 
denote the image of $F\colon \X\times\Z\to\Y$ on the class $\Mo\times \E$. We consider inverse problems as: 
\begin{align}\label{eq:sampling1_bis}
    \text{Given noisy measurements } y = F(x,e) \text{ of } x \in \Mo\text{ and }  e \in \E, \text{ recover } x,
\end{align}
where $e$ represents noise and $x$ is the signal of interest. The goal of solving the inverse problem is to produce an approximation $\hat{x}$ of the true solution $x$.
In order to quantify the error of the approximation $\hat{x} \in \mathcal{X}$ to the sample signal $x\in \Mo$, it is common to assume \cite{pinkus2012n, CoDaDe-08, FoucartRauhutCSbook} that $\mathcal{X}$ is a normed space, or even a Banach or Hilbert space. In the current work, we extend the scope to the more general setting of metric spaces. The use of metrics -- rather than norms -- allows for a more general theory, in which normed vector spaces are a special case. Thus, we equip the sets $\mathcal{X}$, $\mathcal{Y}$ and $\mathcal{Z}$ with metrics $d_\mathcal{X}$, $d_\mathcal{Y}$ and $d_\mathcal{Z}$, respectively, and on each set we consider the topology induced by the respective metric. To avoid certain pathological cases for metric spaces, we make assumptions on the choice of metrics involved. Our first assumption is the following.

\begin{assumption}\label{a:sc} ~
\begin{enumerate}[label=(\roman*)]
       \item\label{it:as1} We assume that the metrics $d_\mathcal{X}$, $d_\mathcal{Y}$, $d_\mathcal{Z}$ are chosen so that the topologies induced by the metrics are second countable (they admit a countable base).
        \item\label{it:as2} 
    We assume that the metrics $d_\mathcal{X}$ and $d_\mathcal{Z}$ satisfy the Heine-Borel property, i.e., that all closed and bounded sets are compact. 
   \end{enumerate}
\end{assumption}
\noindent Here, the first assumption ensures that the topologies induced by the metrics are separable, which is a convenient assumption for both theoretical and practical purposes. In particular, this assumption excludes metrics such as the discrete metric over an uncountable set, as one cannot find a countable base for an uncountable set using this metric.

\noindent For the second assumption, we recall that in metric spaces, a compact subset is automatically closed and bounded; however, the converse does not necessarily hold, unless the metrics satisfy the Heine-Borel property. Note that this \ref{it:as2} is not implied by \ref{it:as1}, as for example the bounded metric $d_\X(x,y) = \min\{\|x-y\|_{\ell^2}, 1\}$ on $\mathcal{X} = \C^N$ satisfies assumption \ref{it:as1} (since $d_\X$ induces the same topology as the Euclidean topology, which is second countable), but not assumption \ref{it:as2} (since $\C^N$ is closed and bounded but not compact with respect to $d_\X$). 
Finally, note that together with second countability, the Heine-Borel property implies in particular that $\mathcal{X}$ is a complete separable metric space, usually referred to as a \textit{Polish space}.

\subsection{Set-valued reconstruction maps}

Now that the preliminaries have been clarified, we return to the problem of solving the inverse problem in \eqref{eq:sampling1_bis}. A \emph{reconstruction mapping} for \eqref{eq:sampling1_bis} is a mapping that takes a measurement $y=F(x,e)$ and returns one or more approximations $\hat{x}$ to the unknown signal $x$. That is, a \emph{reconstruction mapping} is a set-valued function $\varphi\colon \Mt^\E  \rightrightarrows \X$, where $\varphi$ takes values in the power set $2^{\X}$ of the set $\X$. Note that it is necessary to consider set-valued maps, as model-based methods often write the reconstruction mapping as an optimization problem. These problems might not have unique solutions, see e.g., \cite[SI, Sec. 3.B]{comp_stable_NN22} for a few standard examples. Finally, we mention that given a set-valued mapping $\varphi \colon \Mt^\E  \rightrightarrows \X$, we call a function $f \colon \Mt^\E  \to \X$ a \emph{selector} for $\varphi$ if it satisfies $f(y) \in \varphi(y)$ for each $y \in \Mt^\E$.

\noindent Distances between sets in $\X$ will be measured via the Hausdorff distance. Given subsets $A,B \subseteq \X$, their \textit{Hausdorff distance} is
\[
d_{\X}^H(A,B):= \max \Big\{\adjustlimits\sup_{a \in A} \inf_{b\in B} d_{\X}(a,b), \adjustlimits\sup_{b \in B} \inf_{a\in A} d_{\X}(a,b)\Big\}.
\]
The Hausdorff distance satisfies all the properties of a metric only when restricted to the subsets of $\X$ that are bounded (which ensures $d_\X^H$ is finite) and closed (which ensures that $d_\X^H(A,B)=0$ only if $A=B$). With a slight abuse of notation, we will denote $d_\X^H(\{a\},B)$ by $d_\X^H(a,B)$. The Hausdorff distance between a point and a set should not be confused with the usual distance between a point and a set, $\operatorname{dist}_\X(a,B):= \inf_{b\in B} d_\X(a,b)$. Indeed, note that
$\operatorname{dist}_\X(a,B)=\inf_{b\in B} d_\X(a,b)\leq \sup_{b\in B} d_\X(a,b)=d_\X^H(a,B)$.

\noindent We require an extra assumption that relates the forward map $F$ and the model class $\Mo$. In particular, we require that for every $y$, the set of possible true solutions $x$'s is bounded. Throughout the text, we denote by $\pi_1\colon \mathcal{X} \times \mathcal{Z} \to \mathcal{X}, (x,e) \mapsto x$ the projection on the first component.
\begin{assumption}\label{a:compact}
	We assume that  for every $y \in \Mt^\E$ the \emph{feasible set}
		\begin{align}
		F_y:=\pi_1(F^{-1}(y)) = \{x \in \Mo: \exists e\in\E \text{ s.t. } F(x,e)=y\}
		\end{align}
		is bounded.
\end{assumption}

\noindent The feasible set $F_y$ consists of all the candidate solutions $x$'s that are consistent with the measurement $y$, for some realisation of the noise $e$. Note that the definition of $\Mt^\E$ implies that the feasible set is non-empty. Additionally, in Remark \ref{rem:examplecases} we outline assumptions on the forward map, such that the feasible set contains more than one element.\newline

\noindent The previous condition is satisfied, for example, if the model class $\Mo$ is compact, as assumed in e.g., \cite{binev2022optimal}.

\noindent Finally, we clarify our notation for objects in a metric space $(X,d)$. We denote by $B_d(x,r)$ the \textit{closed} ball of center $x\in X$ and radius $r\geq 0$. If $A\subseteq X$, then the closed ball around $A$ of radius $r\geq 0$ is 
\begin{align*}
B_d(A,r) = \{x \in X: \operatorname{dist}_d(x,A) \leq r\} = \bigcup_{x \in A} B_d(x,r)
\end{align*}
where we recall that $\operatorname{dist}_d(x,A) = \inf_{a \in A} d(x,a)$. 

\noindent Finally, the diameter of $A\subseteq X$ is denoted by $\diam_d(A):=\sup\{d(a,a'):a,a'\in A\}$. The subscript $d$ is omitted whenever the metric is unambiguous from context.

\section{Main results}\label{sec:mainres}

This section introduces the inverse problem framework studied in this paper, highlighting the key concepts of an \emph{optimal mapping} and the \emph{kernel size}, which serve as bounds on the reconstruction error. The section is divided into two subsections that both introduce the definitions of an \emph{optimal mapping} and the \emph{kernel size}, but for different measures of accuracy. In particular, $\S$ \ref{ss:wc} presents an analysis based on the \textit{worst-case} error, whereas $\S$ \ref{ss:average} considers a probabilistic model class, where an \textit{average} error is analyzed. The main results, given in Theorems~\ref{thm:optimal_bounds_sup} and \ref{thm:optimal_bounds} for $\S$~\ref{ss:wc} and $\S$ \ref{ss:average}, respectively, bound the \emph{optimality constant} of an optimal mapping in terms of the \emph{kernel sizes} and provide an explicit construction of an optimal reconstruction mapping. Finally, in $\S$ \ref{s:examples} we elaborate on how these quantities relate to common frameworks in the literature. The proofs of the main results are referred to $\S$ \ref{sec:proofmeth}.

\subsection{Worst-case optimality bounds and an optimal map with worst-case noise}\label{ss:wc}
What is the best possible reconstruction error one can achieve for a given model class $\mathcal{M}_1$? This is an old question \cite{micchelli1977survey}, that has been asked many times, and in different settings \cite{CoDaDe-08, fundamental14, binev2022optimal, traonmilin2018stable, keriven2018instance, traonmilin2021theory, ettehad2021instances, foucart2022learning}. In the definition below, we investigate this question for measurements contaminated by worst-case noise. This definition has appeared in \cite{ettehad2021instances, foucart2022learning} for the linear model with additive noise in a normed vector space, and is here generalized to metric spaces. 

\begin{definition}[Optimality constant with worst-case noise]\label{eq:mapconstn3}
The \textit{optimality constant with worst-case noise} of the inverse problem \eqref{eq:sampling1} is 
		\begin{align*}
            c_{\mathrm{opt}}^{\text{w}}(F,\Mo, \E) & = \inf_{\varphi\colon \Mt^{\E} \rightrightarrows \X} \sup_{x\in \Mo} \sup_{e \in \E} \ d_\X^H(x, \varphi(F(x,e))). 
		\end{align*}
		A mapping $\varphi\colon \Mt^{\E} \rightrightarrows \X$ that attains such an infimum is called an \textit{optimal map with worst-case noise}.
\end{definition}

\noindent Given the optimality constant with worst-case noise of an inverse problem \eqref{eq:sampling1}, the question arises on how to upper and lower bound such optimality constant. As shown in Theorem \ref{thm:optimal_bounds_sup}, the upper and lower bounds can be obtained by referring to a constant intrinsic to the problem, its \textit{kernel size with worst-case noise}, as defined below.

\begin{definition}[Kernel size with worst-case noise]\label{def:best_worst_case_rec_error}
	The \textit{kernel size with worst-case noise} of the problem \eqref{eq:sampling1} is
        \begin{align}\label{eq:kersize_det}
		\operatorname{kersize}^{\text{w}}(F,\Mo, \E) &= \sup_{\substack{ (x,e),(x',e')\in\Mo\times\E \text{ s.t. }\\ F(x,e) = F(x',e')} } d_\X(x,x').
		\end{align}
\end{definition}

\begin{remark}
	The kernel size with worst-case noise has been considered under different names in previous work. This includes, but is not limited to, the supremum taken over the measurements $y$ of diameter of the Chebyshev balls of the feasible sets $F_y$, similar to \cite{binev2022optimal}, or the diameter of information in \cite{plaskota1996noisy}. However, as this work focuses on characterizing accuracy bounds for undersampled and ill-posed inverse problems, we opt for referring to the above defined constant as kernel size. In Example 1 in $\S$ \ref{s:examples}, more details on this relation are presented. 
\end{remark}

\noindent The kernel size with worst-case noise gives the maximum distance between any two points $x,x'\in \Mo$ with identical measurement $y = F(x,e) = F(x',e')$ for some noise vectors $e,e'\in \E$. 
It is worth observing that the optimality constant and the kernel size with worst-case noise also can be expressed as follows.
\begin{align}\label{eq:kersizefeas}
c_{\mathrm{opt}}^{\text{w}}(F,\Mo, \E) & = \inf_{\varphi\colon \Mt^{\E} \rightrightarrows \X} \sup_{y\in \Mt^\E} \sup_{x \in F_y} \ d_\X^H(x, \varphi(y)) \nonumber \\
\text{kersize}^{\text{w}}(F,\Mo, \E) &
=\sup_{y \in \mathcal{M}_{2}^{\mathcal{E}}} \mathrm{diam}_{d_\X}(F_y).
\end{align}
This establishes a connection of the optimality constant and the kernel size with worst-case to the feasible sets $F_y$. In our first theorem, we provide an upper and lower bound on the optimality constant in terms of the kernel size with worst-case noise. Moreover, we provide a variational expression for an optimal map with worst-case noise.

\begin{theorem}[Worst-case optimality bounds]\label{thm:optimal_bounds_sup}
   Under Assumptions \ref{a:sc} and \ref{a:compact}, the following holds. 
    \begin{enumerate}[label=(\roman*)]
		\item We have that
        \begin{equation}\label{eq:bounds_kersize}
		\text{kersize}^{\text{w}}(F,\Mo, \E)/2 \leq c_{\text{opt}}^{\text{w}}(F,\Mo, \E) \leq \text{kersize}^{\text{w}}(F,\Mo, \E).
		\end{equation}
		\item The map
        \begin{equation}\label{eq:opt_map_det_ex}
		\Psi(y) = \argmin_{z \in \X}\sup_{x \in F_y} d_\X(x,z) = \argmin_{z \in \X} d_\X^H(z, 
		F_y),
		\end{equation}
has non-empty, compact values and it is an optimal map with worst-case noise.\\ 
	\end{enumerate}
\end{theorem}
\noindent The previous theorem illustrates a fundamental limit for the inverse problem \eqref{eq:sampling1_bis}. Indeed, for \eqref{eq:sampling1_bis} one would hope to find a solution whose error is as close to zero as possible. However, the lower bound in \eqref{eq:bounds_kersize} shows that there is a fundamental constant intrinsic to the problem -- the kernel size with worst-case noise -- such that no worst-case reconstruction error can be made smaller than this constant for all possible choices of $x \in \mathcal{M}_1$. Moreover, note that in the definition of the worst-case optimality constant all potentially set-valued maps are included - in particular, no single-valued selector can obtain a smaller worst-case reconstruction error. Additionally, we can apply the axiom of choice to select a single valued map $\bar{\Psi}$ from the set-valued optimal map in \eqref{eq:opt_map_det_ex}. By the definition of the Hausdorff distance we have that this single-valued selection satisfies $d_{\mathcal{X}}(x,\bar{\Psi}(F(x,e))) \leq \sup_{z \in \Psi(F(x,e))}d_{\mathcal{X}}(x,z)=d_{\mathcal{X}}^H(x,\Psi(F(x,e)))$ for all $x \in \mathcal{M}_1$ and $e \in \mathcal{E}$. Hence, by the minimization property of the optimal map in \eqref{eq:opt_map_det_ex}, the bounds also hold for any single-valued selector.

\begin{remark}
Note that analogous bounds as in \eqref{eq:bounds_kersize} under slightly different assumptions have been obtained in a variety of previous works, including but not limited to \cite{plaskota1996noisy, magaril1991optimal, arestov1986optimal, binev2022optimal}.  However, to the best of our knowledge the characterisation of the set-valued optimal map \eqref{eq:opt_map_det_ex} as compact-valued has not been considered under the same assumptions.
\end{remark}

\subsection{Average optimality bounds and optimal map with average error}
\label{ss:average}
The previous theorem provides fundamental limits when performance is assessed by considering worst-case reconstruction error for a given inverse problem of the form \eqref{eq:sampling1_bis}. Such bounds naturally lead to very pessimistic estimates of performance. This motivates the need for a model that considers the average error of a reconstruction mapping, given a probabilistic model of the data. Below, we will adapt the worst-case setup considered in the previous section to a probabilistic model. Note that in the machine learning literature, this notion of average error is often referred to as the \emph{risk} of a given reconstruction mapping \cite[Ch.\ 3]{shalev2014understanding}\cite[Ch.\ 1]{foucart2022mathematical}.

\noindent We start by introducing the notation. For a topological space $(X,\tau)$, the \emph{Borel $\sigma$-algebra} of $X$, denoted by $\mathcal{B}(X)$, is the $\sigma$-algebra generated by the family $\tau$ of open sets. Elements of $\mathcal{B}(X)$ are called \emph{Borel sets}. If $X$ is a metric space, then $\mathcal{B}(X)$ is the smallest family of sets containing all the open sets that are closed under countable intersections and countable disjoint unions. See, e.g., \cite[Cor.\ 4.16]{guide2006infinite} for further reference.  A \textit{Borel measure} on $X$ is a measure defined on the Borel $\sigma$-algebra $\mathcal{B}(X)$. Borel measurable functions are defined analogously.

\begin{assumption}\label{a:measur}
	We consider the measurable spaces $(\X\times\Z, \mathcal{B}(\X\times\Z))$, $(\Mo\times\E, \mathcal{B}(\Mo\times\E))$, $(\Mt^\E,\mathcal{B}(\Mt^\E))$ and assume that $F\colon \Mo\times\E \to \Mt^\E$ is Borel measurable. We equip  $\X\times\Z$ with a finite Borel measure $\mu$ supported on $\Mo\times\E$, such that $\mu(\Mo\times\E)>0$. We equip $\Mt^\E$ with the pushforward measure $F_*\mu$ given by $(F_*\mu)(E) = \mu(F^{-1}(E))$ for every $E \in \mathcal{B}(\Mt^\E)$. 
\end{assumption}

\noindent Next, we need to define conditional probabilities on $\mathcal{M}_1\times \mathcal{E}$ for the considered measure $\mu$ for different values of $y \in \mathcal{M}_2^{\mathcal{E}}$. To compute such conditional probabilities, we consider a \emph{disintegration} of the measure $\mu$.  
Intuitively, a disintegration of the measure $\mu$ given the mapping $F$, is a family of probability measures $\{\mu^y\}_{y\in \mathcal{M}_2^{\E}}$ such that for each $y\in \mathcal{M}_2^{\E}$, $\mu^y$ is concentrated on the set $F^{-1}(y)$, and which allows us to reconstruct the original measure $\mu$ by integration. Note that these sets $\{F^{-1}(y)\}_{y\in \mathcal{M}_2^{\E}}$ might have measure zero with respect to $\mu$. The concept of disintegration is introduced in the following definition, which is taken from \cite{chang1997conditioning} and adapted to our Assumption \ref{a:measur}. See also \cite[Ch.\ 10]{bogachev2007measurev1} or \cite[Ch.\ 1]{kallenberg1997foundations} for more on disintegration. Moreover, note that in \cite{plaskota1996noisy} a similar tool -- that of a \textit{regular conditional probability density} -- is used in order to prove analogous bounds when restricted to normed spaces and single-valued decoders. 

\begin{definition}\label{def:disintegration}
	A \emph{disintegration} of the measure $\mu$ along the measurable function $F: \mathcal{M}_1\times\mathcal{E} \rightarrow \Mt^\E$ is a family $\{\mu^y\}_{y \in \Mt^\E}$ of probability measures on the measurable space $(\Mo \times \E, \mathcal{B}(\mathcal{M}_1\times\mathcal{E}))$ such that
	\begin{enumerate}[label=(\roman*)]
		\item for $F_*\mu$-almost every $y \in \Mt^{\E}$, $\mu^y$ is a probability measure concentrated on $F^{-1}(y)$, i.e., $\mu^{y}(\mathcal{M}_1\times \mathcal{E}\setminus F^{-1}(y)) = 0$ for $F_*\mu$-almost all $y \in \Mt^{\E}$.
	\end{enumerate} 
	and such that, for each non-negative Borel measurable function $f$ on $(\Mo \times \E, \mathcal{B}(\mathcal{M}_1\times\mathcal{E}))$,
	\begin{enumerate}[label=(\roman*)]
		\setcounter{enumi}{1}
		
		\item the function \[
		y \mapsto \int_{\mathcal{M}_1\times \mathcal{E}} f(x,e) \, d\mu^y(x,e), \quad \text{with}\quad y \in \Mt^\E,\]
		is Borel measurable,
		\item and 
		\begin{equation*}
		\int_{\Mo\times \E} f(x,e) \, d \mu (x,e) = \int_{\Mt^{\E}} \Bigg(\int_{F^{-1}(y)} f(x,e) \, d \mu^y(x,e) \Bigg)\ d(F_*\mu)(y).
		\end{equation*}
	\end{enumerate} 
\end{definition}

\begin{remark}[Example of a disintegration along a projection]
Let $\mu$ be Lebesgue measure on $\mathbb{R}^2$ and let $\pi_1(x,y)=x$ be the projection onto the first coordinate. Then, $\mu$ admits a disintegration with respect to $\pi_1$ given by the family of measures $\{\mu^x\}_{x \in \mathbb{R}}$, where $\mu^x$ is Lebesgue measure on the fiber $\{x\} \times \mathbb{R}$. In this case, the disintegration formula
\[
\int_{\mathbb{R}^2} f(x,y)\, d\mu(x,y)
= \int_{\mathbb{R}} \left( \int_{\mathbb{R}} f(x,y)\, d\mu^x(y) \right) d\pi_{1*}\mu(x,y)
\]
for a non-negative, measurable and integrable function $f$ is the Fubini-Tonelli's theorem (see \cite[Theorem~2.37]{folland1999real}).
\end{remark}

\noindent Note that while we assume that $\mu$ is a finite measure (and not necessarily a probability measure), the definition above requires that $\mu^{y}$'s are probability measures. In $\S$ \ref{ss:existence} we show that a disintegration of $\mu$ always exists and that it is essentially unique in our setting given Assumption \ref{a:Mcompact}, stated below.

\begin{assumption}\label{a:Mcompact}
	We assume that $\Mo$ is compact and we assume that $\E$ is a Borel set.
\end{assumption}
\noindent Note that Assumption \ref{a:Mcompact} implies Assumption \ref{a:compact}, which will therefore be omitted from now on. Additionally, since $\mathcal{M}_1 \times \mathcal{E}$ is a Borel subset of $\mathcal{X} \times \mathcal{Z}$, the Borel $\sigma$-algebra $\mathcal{B}(\mathcal{M}_1\times\mathcal{E})$, that is induced by the subspace topology on $\mathcal{M}_1 \times \mathcal{E}$, coincides with the trace $\sigma$-algebra
$\mathcal{B}(\mathcal{X}\times\mathcal{Z})_{|\mathcal{M}_1\times\mathcal{E}}
=\{ B \cap (\mathcal{M}_1 \times \mathcal{E}) : B \in \mathcal{B}(\mathcal{X} \times \mathcal{Z}) \},$ see e.g., \cite[Section 1.1]{bogachev2007measurev1}.

\begin{remark}\label{rem:examplecases}
	As an example for an inverse problem as in \eqref{eq:sampling1_bis} satisfying the Assumptions \ref{a:compact} and \ref{a:Mcompact}, we consider a linear forward model with additive noise $F(x,e)=Ax+e$ for $x \in \mathcal{M}_1 \subset \mathcal{X}= \mathbb{C}^N$ and $e \in \mathcal{E} \subset \mathcal{Z}= \mathbb{C}^m$ with $m,N \in \mathbb{N}$ and $A \in \mathbb{C}^{m\times N}$. For many relevant applications the linear map $A \in \mathbb{C}^{m\times N}$ has a non-trivial null space $\mathcal{N}(A):=\{x \in \mathbb{C}^N: Ax=0\} \neq \{0\}$ or is ill-conditioned. The linear map typically has a non-trivial null space for undersampled inverse problems. In the case that $A$ is ill-conditioned, numerically the condition $\mathcal{N}(A) \neq \{0\}$ can in principle be approximated by setting the singular values around or below machine precision to zero, which results in $\mathcal{N}(A) \neq \{0\}$. If the null space is non-trivial, the feasible sets $F_y$ can be non-empty. As a subset of the bounded set $\mathcal{M}_1$, the feasible sets are bounded and by the Heine-Borel property of $d_{\mathcal{X}}$, in Assumption \ref{a:sc}, are also compact. In this case the feasible sets satisfy Assumption \ref{a:compact}. Thus, the kernel size with worst-case noise can be non-zero. Assumption \ref{a:Mcompact} is likely to be satisfied in applications. The condition that data sets or model classes, $\mathcal{M}_1 \subseteq \mathbb{C}^N$, are bounded, ensures that the set of admissible signals remains within physically or numerically reasonable limits. Additionally, if $\mathcal{M}_1$ is considered to be data, it can be assumed to be closed, since they include boundary points as part of the domain. Thus, it is natural to treat  $\mathcal{M}_1$ as closed, since excluding boundary points would not align with how physical data is collected or interpreted. Finally, many applications involve quantities that vary continuously (or piecewise continuously) with respect to their underlying parameters. This continuity can imply that boundary points do not represent any abrupt change in the physical state and in such cases can be seamlessly included in  $\mathcal{M}_1$. This supports modelling $\mathcal{M}_1$ as a closed set, as no artificial exclusion of boundary points is required. Hence, assuming that $\mathcal{M}_1$ is closed and bounded is consistent with typical modelling assumptions. By the Heine-Borel property of $d_{\mathcal{X}}$, from Assumption \ref{a:sc}, the set $\mathcal{M}_1$ is compact, as it is closed and bounded, satisfying Assumption \ref{a:Mcompact}. In probabilistic settings, the set of noise vectors $\mathcal{E}$ is assumed to be Borel measurable, satisfying Assumption \ref{a:Mcompact}. 
\end{remark}

\noindent Next, let 
\begin{equation*}
r_\varphi: \Mo \times \E \to [0,+\infty], \quad r_\varphi(x,e) = d_\X^H(x,\varphi(F(x,e)))
\end{equation*}
denote the \textit{residual map} of a given reconstruction mapping $\varphi \colon \Mt^\E \rightrightarrows \X$ and let 
	\begin{equation}\label{eq:C_def}
		\mathcal{C}\coloneqq  \{\varphi\colon \mathcal{M}_2^\E \rightrightarrows \X: r_\varphi \text{ is Borel measurable}\}.
	\end{equation}
denote the set of all reconstruction mappings with a Borel measurable residual map. 
For $p \in [1,\infty]$ and a given a reconstruction map $\varphi \colon \Mt^\E \rightrightarrows \X$, we denote its $p$th order error by 
\begin{equation}\label{eq:error_p}
\mathrm{Err}^{\mathrm{a}}(\varphi,p) = \begin{cases} 
\Bigg(\int_{\Mo\times \E} d_\X^H\Big(x,\varphi(F(x,e))\Big)^p \ d\mu(x,e) \Bigg)^\frac{1}{p} &\text{for } 1\leq p < \infty \\
\essup_{(x,e) \in \Mo\times\E} \ d_\X^H(x, \varphi(F(x,e))) &\text{for } p=\infty
\end{cases}
\end{equation}
where the essential supremum for $p=\infty$ is taken with respect to the measure $\mu$. Concretely, for $p\in [1,\infty]$ this quantity is the $L^p(\mathcal{M}_1\times\mathcal{E}, \mu)$ norm of the residual function $r_\varphi$.

\noindent This definition of reconstruction error generalizes the usual $L^p$-norms to include a set-valued mapping and a metric $d_{\mathcal{X}}$ that is not necessarily induced by a norm in the single-valued case. However, the proposed setup includes many standard expectations as special cases. Indeed, let $\varphi$ be a single-valued map and let $d_{\mathcal{X}}$ be given by an $\ell^q$ norm. Then, if we choose $p=q= 1$, we recover the expected absolute error, whereas for $p=q=2$ we find the expected root squared error. Moreover, for $p=\infty$, the choice of metric $d_{\mathcal{X}}$ corresponds to the choice of the loss function in a machine learning setting \cite[Ch.\ 3]{shalev2014understanding}\cite[Ch.\ 1]{foucart2022mathematical}. One can also view the case where $p=\infty$ as a weak form of worst-case error where all irregularities on sets of measure zero are ignored.  

\noindent With the error term in \eqref{eq:error_p} defined, we can now establish the notion of optimality constant with average error and kernel constant with average error. 
\begin{definition}[Optimality constant with average error]\label{def:optimality_constant}
	For $p \in [1,\infty]$ the \textit{optimality constant with average error} of order $p$ for the inverse problem \eqref{eq:sampling1_bis} is 
	\begin{align}\label{eq:optconst}
		c_{\mathrm{opt}}^{\text{a}}(F,\Mo,\E,  p) = \inf_{\varphi \in \mathcal{C}} \mathrm{Err}^{\mathrm{a}}(\varphi, p). 
		\end{align}
A map $\varphi \in \mathcal{C}$ attaining the infimum in \eqref{eq:optconst} for a given $p$ is called an \textit{optimal map with average error} of order $p$.
\end{definition}

\begin{definition}[Average kernel size]\label{def:kersize}
The average kernel size of the inverse problem \eqref{eq:sampling1} for $p \in [1,\infty)$ is given by
	\begin{align*}
		\operatorname{kersize}^{\text{a}}(F,\Mo, \E, p) &= \Bigg( \int_{\Mt^{\E}} \int_{F^{-1}(y)}  \int_{F^{-1}(y)} d_\X(x,x')^p \, d\mu^y(x,e) \ d\mu^y(x',e') \ d(F_*\mu)(y) \Bigg)^\frac{1}{p},
	\end{align*}
	and for $p = \infty$ it is
    \begin{align}\label{eq:kersize_pinf}
		\operatorname{kersize}^{\text{a}}(F,\Mo, \E, \infty) &= \essup_{y \in \Mt^{\E}} \essup_{\substack{(x,e) \in F^{-1}(y)\\ (x',e') \in F^{-1}(y)}} d_\X(x,x'),
	\end{align}
	where the left and right essential suprema are taken with respect to $F_*\mu$ on $\Mt^\E$ and $\mu^y$ on $F^{-1}(y)$, respectively.
\end{definition}

\noindent Intuitively, the average kernel size measures the average distance between the $x$-components of  $(x,e), (x',e') \in\Mo\times \E$ that have the same measurement $F(x,e) = F(x',e')=y$ in the case of $p \in [1,\infty)$, and the maximum of such distances (up to negligible sets) when $p = \infty$.

\noindent Our main theorem in this section is stated below. It bounds the average optimality constant in terms of the average kernel size. It also ensures the existence of an optimal map, provides a variational expression for one such map and establishes some regularity properties of this mapping. 

\begin{theorem}[Average optimality bounds]\label{thm:optimal_bounds}
	Under Assumptions \ref{a:sc}, \ref{a:measur}, \ref{a:Mcompact}, the following holds for every $p \in [1,\infty]$,
	\begin{enumerate}[label=(\roman*)]
		\item We have that
			\begin{equation*}
			\operatorname{kersize}^{\text{a}}(F,\Mo, \E, p)/2 \leq c_{\mathrm{opt}}^{\text{a}}(F,\Mo,  \E,  p)  \leq \operatorname{kersize}^{\text{a}}(F,\Mo, \E, p).
		\end{equation*}
	
	\item Consider $p \in [1,\infty)$, and let $\Psi\colon \mathcal{M}_2^\E\rightrightarrows \X$ be given by 
    \begin{align}\label{eq:optimal_map_inspiration12}
		\Psi(y) &= \argmin_{z \in \X} \int_{F^{-1}(y)} d_\X(x,z)^p \ d\mu^y(x,e).
	\end{align}
	Then we have the following: $\Psi$ has compact values, it is measurable and it admits a measurable selector. Moreover, $\Psi$ is an optimal map with average error of order $p$.
	\item Consider $p = \infty$, and let $\Psi\colon \mathcal{M}_2^\E\rightrightarrows \X$ be given by 
	
		\begin{align}\label{eq:optimal_map_inspiration1}
		\Psi(y) &= \argmin_{z \in \X}\essup_{(x,e) \in F^{-1}(y)} d_\X(x,z). 
	\end{align}
	Then we have the following: $\Psi$ has compact values, it is measurable and it admits a measurable selector. Moreover, $\Psi$ is an optimal map with average error of order $p$.
	\end{enumerate}
\end{theorem}
\begin{remark}[Measurability of set-valued mappings] There exist several notions of measurability for set-valued mappings. The precise definition of the term used in the theorem above can be found in Definition \ref{def:measurability}. A measurable selector from a set-valued mapping $\Phi : \mathcal{Y} \rightrightarrows \mathcal{X}$ between measurable spaces is a measurable function $\phi: \mathcal{Y} \rightarrow \mathcal{X}$ such that $\phi(y) \in \Phi(y)$ for all $y \in \mathcal{Y}$, \cite[$\S$.\ 18.3]{guide2006infinite}.
\end{remark}
\noindent A key finding in Theorem \ref{thm:optimal_bounds} is that the possibly set-valued $\Psi$ admits a measurable selector, i.e. a \textit{single-valued} function $\psi:\Mt^\E \to \X$ which is Borel measurable and such that $\psi(y) \in \Psi(y)$ for every $y$. Thus, Theorem \ref{thm:optimal_bounds} not only provides fundamental limits for the reconstruction error of an inverse problem, but also gives a (variational) expression for a mapping that is optimal. Moreover, such a reconstruction mapping satisfies some regularity properties, which one usually hopes for when solving an inverse problem. One can argue that measurability is still a weak condition, but in order to obtain stronger conditions (such as continuity of the reconstruction), one would need to impose stronger assumptions on the problem.  Moreover, note that in the definition of the average optimality constant all potentially set-valued maps are included. Hence, the bounds also hold for any measurable single-valued selector.

\section{Examples}
\label{s:examples}

In this section, we present three examples from research areas in inverse problems and briefly illustrate their connection to the framework developed in this paper. Additionally, we provide examples where the upper and lower bounds in Theorem \ref{thm:optimal_bounds_sup} $(i)$ and the upper bound in Theorem \ref{thm:optimal_bounds} $(i)$ are attained and, hence, show that the bounds are sharp.

\subsection*{Example 1: Linear inverse problems with the robust null-space property}
The robust null space property from compressed sensing was first introduced as a necessary condition for uniform recovery of sparse vectors from linear measurements. See \cite{FoucartRauhutCSbook} for a historical overview. Later it has been extended to a wide range of model classes $\mathcal{M}_1$, including low-rank \cite{NSP_low_rank} and sparsity in levels \cite{bastounis2017absence} models. A very general version of the property appears in \cite{fundamental14} which considers general sets $\mathcal{M}_1$. We now recall a specialized version of the property from \cite{fundamental14}, and, then, we relate to this to kernel size with worst-case noise. 

\noindent Let $\mathcal{X}, \mathcal{Y}$ and $\mathcal{Z}$ be vector spaces, with $\mathcal{Y}=\mathcal{Z}$, and let $A\colon \mathcal{X} \to \mathcal{Y}$ be a linear mapping that is onto $\mathcal{Y}$. We consider the additive linear model $F(x,e) =Ax+e$ for $(x,e) \in \Mo\times\E \subset \X\times\Z$. Let $\vertiii{\cdot}_{\X_1}$ and $\vertiii{\cdot}_{\X_2}$ be norms on $\mathcal{X}$, and let $\vertiii{\cdot}_{\Y}$ be a norm on $\mathcal{Y}$ (slightly weaker conditions are used in \cite{fundamental14}). Furthermore, for the model class $\Mo \subset \X$ we denote the set $\mathcal{M}_1-\mathcal{M}_1:=\{x-x' : x,x'\in \mathcal{M}_1\}$, and let $\mathrm{dist}_{\X_2}(h,\mathcal{M}_1-\mathcal{M}_1)=\inf_{z \in \mathcal{M}_1-\mathcal{M}_1} \vertiii{z-h}_{\X_2}$. Then, the linear mapping $A$ is said to satisfy the \textit{robust null-space property} (rNSP) with constants $D_1, D_2 > 0$ with respect to the set $\Mo$, if 
\begin{equation}\label{eq:rNSP}
    \vertiii{h}_{\X_1} \leq D_1 \mathrm{dist}_{\X_2}(h, \mathcal{M}_1-\mathcal{M}_1) + D_2 \vertiii{A(h)}_{\Y} \quad \text{for all } h \in \mathcal{X}. 
\end{equation}

\noindent In Proposition \ref{prop:rnspkers} we assume that Assumption \ref{a:Mcompact} holds, i.e., that $\mathcal{M}_1$ is compact. For linear inverse problems, this assumption is often used on the noise level, whereas the set $\mathcal{M}_1$ is often an unbounded set (sparse vector, matrices with low rank etc.). However, when modelling any practical application it is not unreasonable to assume that the set of interest is compact or closed and bounded. Now under the assumption that the noise is bounded by $\mathrm{diam}_{d_{\Y}}(\mathcal{E}) = \sup_{e,e' \in \E}\vertiii{e-e'}_{\Y} = \eta$ in Proposition \ref{prop:rnspkers} we show that the kernel size with worst-case noise, with respect to the metrics induced by $\vertiii{\cdot}_{\X_1}$ and  $\vertiii{\cdot}_{\Y}$, is bounded by 
\begin{equation}\label{eq:up_bound_lin}
 \mathrm{kersize}^{\mathrm{w}}(F,\mathcal{M}_1, \mathcal{E}) \leq D_2\eta. 
\end{equation}
In this sense Proposition \ref{prop:rnspkers} provides a relation between the rNSP and the kernel size with worst-case noise.
\begin{proposition}\label{prop:rnspkers}
Let $\eta \geq 0$. Let $\mathcal{X}, \mathcal{Y}$ be normed vector spaces with norms $\vertiii{\cdot}_{\X_1}$ and $\vertiii{\cdot}_{\X_2}$ on $\mathcal{X}$, and  $\vertiii{\cdot}_{\Y}$ on $\Y$. Let $\Mo \subseteq \X$ satisfy Assumption \ref{a:Mcompact} and let $\E \subset \Y$ with $\mathrm{diam}_{d_{\Y}}(\mathcal{E}) = \eta$. Let $F(x,e) =Ax+e$ for $(x,e) \in \Mo\times \E$, where $A\colon \mathcal{X} \to \mathcal{Y}$ is a linear map onto $\mathcal{Y}$ and that satisfies the rNSP with constants $D_1, D_2 \geq 0$ with respect to $\Mo$. Then, the following holds:
\begin{enumerate}
    \item $A$ is injective on $\Mo$ and this is equivalent to
    \begin{align*}
    (\Mo - \Mo) \cap \mathcal{N}(A) = \{0\},
    \end{align*}   
    where $\mathcal{N}(A) =\{x \in \X: Ax=0\}$.
    \item The kernel size with worst-case noise can be bounded from above as
    \begin{align*}
    \operatorname{kersize}^{\text{w}}(F, \Mo, \E) \leq D_2 \eta.
    \end{align*}
\end{enumerate}
\end{proposition}
\noindent Note that if we take $\eta=0$, i.e. a noiseless forward model $\mathcal{E} = \{0\}$, we have that $\mathrm{kersize}^{\mathrm{w}}(F,\mathcal{M}_1, \mathcal{E}) = 0$.

\subsection*{Example 2: Optimal learning, Chebyshev centers and Gelfand widths}
Recently the notion of optimal learning was introduced in \cite{binev2022optimal}. Herein, one considers the setting in which the set $\mathcal{X}$ is a Banach space with norm $\|\cdot\|_{\mathcal{X}}$ and $\mathcal{M}_1$ is a compact subset of $\mathcal{X}$. 
To sample an element $x \in \mathcal{M}_1$, one uses $m$ linear functionals  $\lambda_1,\ldots, \lambda_m \in\mathcal{X}^*$ from the dual space of $\mathcal{X}$. Now, for a given $x \in \mathcal{M}_1$, let $y = (\lambda_1(x), \ldots, \lambda_m(x)) \in \C^m$ and let
\[K_{y} = \{x \in \mathcal{M}_1 : \lambda_i(x)= y_i , i \in \{1,\dots,m\}\}\]
denote the set of solutions which are data-consistent. The above measurement model is noiseless, so define for convenience $\mathcal{M}_2 = \{y=(\lambda_1(x),\ldots,\lambda_m (x)) : x \in \mathcal{M}_1\}$ as the set of noiseless measurements.

\noindent Now, let $B_{\mathcal{X}}(z,r)$ denote the closed ball in $\mathcal{X}$ with center $z$ and radius $r$. For a compact set $\mathcal{S} \subset \mathcal{X}$, the Chebyshev radius of $\mathcal{S}$ is given by 
\[
R_{\mathcal{X}}(\mathcal{S}) \coloneqq \inf \{r>0 : \mathcal{S}\subset B_{\mathcal{X}}(z,r) \text{ for some }z\in \mathcal{X}\}.
\]

\noindent In \cite{binev2022optimal} the quantity $R_{\mathcal{X}}(K_{y})$ is defined as the \emph{optimal recovery rate} for a given $x \in \mathcal{M}_1$, with $y=(\lambda_1(x),\ldots,\lambda_m(x))$. This quantity is zero if  $K_{y}$ is a singleton, however, here we follow \cite{binev2022optimal} and focus on the cases for which $R_{\mathcal{X}}(K_{y}) > 0$.  

\noindent While the framework developed in this paper does not extend to general Banach spaces, it does cover the special cases of $\mathcal{X}=\C^N$ with metric induced by a norm $\|\cdot\|_{\mathcal{X}}$. Indeed, by letting $A \in \C^{m\times N}$ be the matrix given by the linear functionals $\lambda_i$, taking $\mathcal{E} =\{0\}$, and using the linear additive model $F(x,e) = Ax+e$, as $K_y = F_y$ it is straightforward to see that $R_{\X}(K_y) = \tfrac{1}{2}\mathrm{diam}_{\X}(F_y)$, since the minimal enclosing ball of a bounded set has radius equal to half its diameter when centered at a Chebyshev center. By taking the supremum over all $y \in \mathcal{M}_2$, we obtain of $\sup_{y\in \Mt} 2R_{\X}(K_y) = \sup_{y\in \Mt} \mathrm{diam}_{\X}(F_y)$, which coincides with the kernel size with worst-case noise, as characterized in \eqref{eq:kersizefeas}. To further analyze the noisy setting, we refer to \cite{binev2022optimal}.

\noindent The framework proposed in this work is closely related to the concept of Chebyshev centers. In fact, the optimal maps $\Psi$ defined in \eqref{eq:opt_map_det_ex} and \eqref{eq:optimal_map_inspiration12} correspond to the problem of finding the Chebyshev center or the $p$-center (see \cite{novak1989stochastic}) of the set $F_y$ of candidate solutions. The setting proposed in this paper guarantees that such centres exist, so that the optimal maps are well-defined; however, Chebyshev centers may not exist in general, and other sufficient and necessary conditions for their existence can be found in the literature (see \cite{amir1982existence}, \cite{amir1984chebyshev}, \cite{ward1974chebyshev}, \cite{rao2002chebyshev}).

\noindent In this work, we have kept the forward model $F$ fixed. However, in the special case of a linear forward model, it is reasonable to also consider the question of how well we can perform given a fixed budget of $m$ linear functionals. This question is well studied in the literature \cite{pinkus2012n}, and is given by the Gelfand width of $\mathcal{M}_1$, 
\[
w^m(\mathcal{M}_1)\coloneqq  \inf_{A\in \C^{m\times N}} \sup_{x\in \mathcal{M}_1}R_{\mathcal{X}}(K_{Ax}).
\]
We note that this presents a lower bound for $c_{\mathrm{opt}}^{\mathrm{w}}(A, \mathcal{M}_1, \{0\})$ in the noiseless setting. We do not consider the question of what the optimal forward model would be any further in this work, but remark that it is an interesting question worth investigating. 

\subsection*{Example 3: Bayesian Inverse Problems}

Bayesian inverse problems \cite{Arr19, stuart2010inverse} take a probabilistic approach. Instead of using a measurement $y$ to output a \textit{single} candidate solution $x$, the goal of a Bayesian approach is to use a measurement $y$ to output a \textit{distribution} on the possible solutions $x$'s.

\noindent The following proposition of Theorem \ref{thm:optimal_bounds} demonstrates how Bayesian inverse problems can be related to the framework presented in this work. The proposition shows that the posterior mean estimator is an optimal map in the average-case for $p=2$ and for metrics $d_\X$ that are induced by an inner product. 
\begin{proposition}[Bayesian Problems in the current framework]\label{prop:Bayes}
Let $(\Omega, \mathcal{F}, \mathbb{P})$ be a probability space.
Let $(\X,d_\X), (\Y,d_\Y), (\Z, d_\Z)$ and $\Mo, \E$, $F : (\X \times \E, \mathcal{B}(\X \times \E)) \to (\Mt^\E, \mathcal{B}(\Mt^\E))$ and
\[
(\mathbf{x},\mathbf{e}) : (\Omega, \mathcal{F}, \mathbb{P}) \to (\X \times \E, \mathcal{B}(\X \times \E))
\]
be a random variable with (joint) distribution $\mu := \mu_{\mathbf{x},\mathbf{e}} = \mathbb{P}[(\mathbf{x},\mathbf{e}) \in \cdot]$ satisfying Assumptions \ref{a:sc}, \ref{a:measur} and \ref{a:Mcompact}. Let $(\Y,d_\Y)$ be a Polish space. Define the random variable $\mathbf{y} := F(\mathbf{x},\mathbf{e}) = F \circ (\mathbf{x},\mathbf{e})$ whose law is given by $\mu_{\mathbf{y}} = F_* \mu = \mathbb{P}[\mathbf{y} \in \cdot]$. Denoting by $\pi_1 : \X \times \E \to \X$ the coordinate projection, the following holds:
\begin{enumerate}
    \item The Bayesian prior distribution of $\mathbf{x}$ is the distribution of $\mathbf{x}$ and can be obtained as $\mu_{\mathbf{x}} = \pi_{1*} \mu$;
    \item There exist likelihood distributions $\{\mu_{\mathbf{y} \mid \mathbf{x} = x}\}_{x \in \X}$ of $\mathbf{y}$ given $\mathbf{x}$;    
    \item The Bayesian posterior distributions $\{\mu_{\mathbf{x} \mid \mathbf{y} = y}\}_{y \in \Mt^\E}$ can be obtained by the pushforward with respect to $\pi_1$ of the disintegration
    $\{\mu^y\}_{y \in \Mt^\E}$ of $\mu$ along $F$ as: $\mu_{\mathbf{x} \mid \mathbf{y} = y} = \pi_{1*} \mu^y$;
    \item If the metric $d_{\X}$ is replaced with a metric induced by an inner product $\langle \cdot, \cdot \rangle$ and $p = 2$, then the optimal map $\Psi$ from  Theorem \ref{thm:optimal_bounds} part (ii) coincides with the posterior mean and $\Psi(y) = \mathbb{E}_{\mu_{\mathbf{x} \mid \mathbf{y}=y}}[\mathbf{x}]$ for every $y \in \Mt^\E$.
\end{enumerate}
\end{proposition}
\noindent As Proposition \ref{prop:Bayes} is intended only to be a brief example, the relevant definitions and its proof are presented in $\S$ \ref{sec:proofmeth}. This result illustrates how the framework introduced in this chapter is capable of capturing not only inverse problems in their classical sense,
but also in the Bayesian approach.

\subsection*{Example 4: The kernel sizes are sharp bounds to the optimality constants}
In the following, we provide two examples that show that the upper and lower worst-case bounds in Theorem \ref{thm:optimal_bounds_sup} $(i)$ can be attained and, hence, are sharp. Then, we provide another example showing that the upper bound is attained in Theorem \ref{thm:optimal_bounds} $(i)$ and, hence, is also sharp. We first provide two simple examples that obtain the upper and lower worst-case bounds in Theorem \ref{thm:optimal_bounds_sup} $(i)$. In the following, for fixed $n \in \mathbb{N}$, $\{e_1, \ldots, e_n\}$ denotes the canonical basis of $\mathbb{R}^n$. 
\begin{enumerate}
    \item \textit{Sharpness of the lower bound in Theorem \ref{thm:optimal_bounds_sup} $(i)$}: Let $(\X, d_\X) = (\mathbb{R}^2, \|\cdot\|_2)$, $(\Y, d_\Y) = (\E, d_\E) = (\mathbb{R}, |\cdot|)$, $\Mo = \{-e_1, e_1\}$, $\E = \{0\}$, and
    $F : \mathbb{R}^2 \times \mathbb{R} \rightarrow \mathbb{R}$ with $F((x_1, x_2), e) = x_2$, so that $\Mt^\E = \{0\}$. Then, as $F^{-1}(0) = \Mo$, we have that
    \begin{align*}
    \operatorname{kersize}^{\text{w}}(F, \Mo, \E) = \|e_1-(-e_1)\|_2=2, \quad c_{\mathrm{opt}}^{\text{w}}(F, \Mo, \E) = 1,
    \end{align*}
    as the optimal map with worst-case noise is given by $\Psi(0) = \argmin_{z \in \mathbb{R}^2} \sup_{x \in \Mo} \|z-x\|_2 = 0$.
    \item \textit{Sharpness of the upper bound in Theorem \ref{thm:optimal_bounds_sup} $(i)$}: Consider the same setup as above, but replace $\X$ with $(\X, d_\X) = (\Mo, \|\cdot\|_2)$ and as $\Mo \subset \mathbb{R}^2$ restrict $F$ by $F : \Mo \times \mathbb{R} \rightarrow \mathbb{R}$ with $F((x_1, x_2), e) = x_2$. Then, we have that
\begin{align*}
    \operatorname{kersize}^{\text{w}}(F, \Mo, \E) = 2 = c_{\mathrm{opt}}^{\text{w}}(F, \Mo, \E),
    \end{align*}
    as the optimal maps with $\Psi_i(0) = \argmin_{z \in \Mo} \sup_{x \in \Mo} \|z-x\|_2$ are given by $\Psi_i: \Mt^\E \rightarrow \Mo$ for $i = 1,2,3$, where $\Psi_1(0) = e_1$, $\Psi_2(0) = -e_1$ and $\Psi_3(0) = \{-e_1, e_1\}$.
\end{enumerate}
\noindent For the average case by providing examples attaining the bounds, we can also show that the upper bound in Theorem \ref{thm:optimal_bounds} $(i)$ is sharp. Consider the same setup as in part $(1)$ above. Let $\alpha \in [0,1]$ be an arbitrary parameter and equip $\mathbb{R}^2 \times \E$ with the measure $\mu = \mu_\alpha$ given by $\mu_\alpha = \bigl(\alpha \delta_{-e_1} + (1-\alpha)\delta_{e_1}\bigr) \otimes \delta_0$, which is supported on $\Mo \times \E = \{-e_1, e_1\} \times \{0\}$. For $p = 1$ and as $F^{-1}(0) = \Mo$, we have that the average kernel size is  \begin{align*}
    \operatorname{kersize}^{\text{a}}(F, \Mo, \E, 1) = \alpha^2 \cdot 0 + 2 \alpha (1-\alpha) 2 + (1-\alpha)^2 \cdot 0 = 4\alpha(1-\alpha),
    \end{align*}
    and the optimality constant and the optimal map with average error $\Psi: \Mt^\E  \rightrightarrows \mathbb{R}^2$, obtained by Theorem \ref{thm:optimal_bounds} $(ii)$ as $\Psi(0) = \argmin_{z \in \mathbb{R}^2} \left(\alpha \|z+e_1\|_2+(1-\alpha) \|z-e_1\|_2\right)$, are
    \begin{align*}
    c_{\mathrm{opt}}^{\text{a}}(F, \Mo, \E, 1) =
    \begin{cases}
        2\alpha, & \text{if } \alpha \in [0, \tfrac{1}{2}), \\
        1, & \text{if } \alpha = \tfrac{1}{2}, \\
        2(1-\alpha), & \text{if } \alpha \in (\tfrac{1}{2}, 1],
    \end{cases}
     \text{ and,  } 
    \Psi(0) =
    \begin{cases}
        e_1, & \text{if } \alpha \in [0, \tfrac{1}{2}), \\
        \{(t,0): t \in [-1,1]\}, & \text{if } \alpha = \tfrac{1}{2}, \\
        -e_1, & \text{if } \alpha \in (\tfrac{1}{2}, 1].
    \end{cases}
    \end{align*}
For $\alpha = \tfrac{1}{2}$, any single-valued selector $\psi: \Mt^\E \rightarrow \mathbb{R}^2$ of $\Psi$ is an optimal map.
For $\alpha = \tfrac{1}{2}$ the upper bound in Theorem \ref{thm:optimal_bounds} $(i)$ is attained, as $c_{\mathrm{opt}}^{\text{a}}(F, \Mo, \E, 1)=\operatorname{kersize}^{\text{a}}(F, \Mo, \E, 1)=1$.

\section{Relation to previous work}

As an ongoing topic of research in many areas of mathematics, undersampled and ill-posed inverse problems have been studied in many different areas. As a non-extensive list of examples, they have been studied from a theoretical standpoint \cite{engl2014inverse}, from a statistical perspective \cite{o1986statistical}, by using iterative deep neural networks \cite{adler2017solving} and in applied fields, such as radio tomography of the ionosphere \cite{garcia2008radio}. For solving undetermined and ill-posed inverse problems often regularization methods are utilized -- here M. Benning et al. give a comprehensive review of modern methods \cite{benning2018modern}. Different noise models have been studied, and while an assumption of additive noise appears ubiquitously \cite{hofinger2009convergence, kaipio1999inverse, bertero2021introduction}, often multiplicative noise models can be relevant in applications \cite{shi2008nonlinear, zhao2014new, aubert2008variational, huang2009new}. The authors of \cite{scarlett2022theoretical} note that an open challenge in this research field is optimization guarantees with milder assumptions, which are provided in this work. To the best of the authors' knowledge fundamental accuracy bounds for undersampled and ill-posed inverse problems with multiplicative noise have not been produced. Moreover, despite the existing extensive amount of research, there is little to be found on fundamental accuracy bounds of approximate set-valued solutions to undersampled and ill-posed inverse problems. 

\textbf{Undersampled inverse problems and AI hallucinations:} In \cite{SIREV_paper} we provide sufficient conditions for AI hallucinations to occur due to a non-trivial kernel of the forward operator. For example, instabilities, such as false positives, false negatives, and especially AI hallucinations, have been an issue in the fastMRI challenge \cite{fastmri20} and in microscopy \cite{bel19, hoff21}. In \cite{burger2024learning} M. Burger and T. Roith outline that fundamental performance and accuracy limits for data-driven approaches, such as deep learning, for solving ill-posed inverse problems do not exist. Moreover, they outline that the lack of such meaningful error estimates in the context of medical imaging and undetermined inverse problems is demonstrated using AI hallucinations in recent studies \cite{PNAS_paper, bhadra2021hallucinations, SIREV_paper}. In this work we provide the missing link between a non-trivial kernel and optimal recovery.

\textbf{(Set-valued) Decoders arising in Deep Learning:}
In the work \cite{daubechies2022nonlinear}, the ability of deep neural networks to nonlinearly approximate functions, and thus also decoders, is investigated. A range of results on how the resulting decoder may be constituted is presented by M. Unser in \cite{unser2020unifying}. From an application-based perspective, a detailed overview of deep learning in inverse problems and stability of robustness for deep learning is given in \cite{mccann2017convolutional, Arr19, lucas2018using}. Compared to standard methods for solving inverse problems, partially data-driven approaches, such as from L. Borcea, J. Garnier et al. \cite{borcea2023waveform}, and fully data-driven approaches using deep learning, such as \cite{Bo-18, Str-18, bel19}, have reported superior accuracy. This can potentially lead to instabilities, which is highlighted by V. Antun et al. in \cite{PNAS_paper} and in \cite{SIREV_paper}. In fact, a variety of research has established that artificial intelligence techniques based on deep learning are unstable, firstly in image classification \cite{eykholt2018robust, kurakin17, DezFa-16, nguyen2015deep, SzZ-14}, and later in applications ranging from audio and speech recognition \cite{hidden_voice, Carlini18, zhang2017dolphinattack} to natural language processing \cite{liangFoolText} and automatic diagnosis in medicine \cite{Science_adv}. Moreover, many DL based approaches implicitly include set-valued functions, where examples include, but are not limited to, deep ensembles \cite{lakshminarayanan2017simple} or model-based probabilistic conditional diffusion models as in \cite{liu2023dolce}. In fact, any probabilistic DL model used to solve an inverse problem that uses sampling from a distribution or predicts parameters of a distribution -- see the work of J. Gawlikowski \cite{gawlikowski2023survey} for an introduction on uncertainty quantification and probabilistic DL models -- can be considered to be a set-valued decoder. 

\textbf{Optimal (set-valued) decoders for undersampled inverse problems:}
The study of optimal and near-optimal recovery in inverse problem has always been a central question, and many single-valued near-optimal mappings have been proposed and analyzed \cite{fundamental14, binev2022optimal, plaskota1996noisy, traonmilin2018stable, FoucartRauhutCSbook}.  A survey of early works on optimal recovery is given in  \cite{micchelli1977survey}.  For Compressed Sensing (CS) C. Poon provides sampling conditions and error bounds for near-optimal recovery in \cite{poon2015role}. The case of set-valued optimal decoders has been considered for obtaining the worst-case bounds on Banach spaces in \cite{arestov1986optimal} and extended to metric spaces in \cite{magaril1991optimal}. As presented in Example 3, the Bayesian approach to inverse problems aims at recovering a distribution valued decoder \cite{Arr19, stuart2010inverse}. Here there exist a posteriori accuracy bounds in the case of normed spaces \cite{leonov2016locally, leonov2012posteriori, wang2020general}.

\textbf{Accuracy bounds in approximation theory:} Approximation theory aims to provide accuracy bounds and theoretical guarantees for function approximation tasks under specific assumptions. However, in \cite{adcock2021gap} B. Adcock and N. Dexter state that despite the ability of deep neural networks to approximate functions relevant to scientific computing, a large gap between expressivity and practical performance remains.  An in-depth treatment of sparse polynomial approximations of high dimensional functions is given by B. Adcock, S. Brugiapaglia and C. Webster in \cite{adcock2022sparse}. Fundamental accuracy bounds, the Gelfand Widths, have been established in \cite{pinkus2012n} for the noiseless linear setting. A framework for obtaining the best $k$-term approximation and corresponding accuracy bounds is established \cite{CoDaDe-08}. In \cite{cohen2022optimal} the authors propose a framework to measure the optimal performance for nonlinear methods of approximation. The notion of optimal learning in a noisy setting and upper worst-case accuracy bounds based on the Chebyshev radius are presented in \cite{binev2022optimal}. B. Adcock et al. \cite{adcock2017breaking} propose a framework for accuracy bounds extending the classical assumption of CS. Generalized instance optimality and the corresponding accuracy bounds have been presented in \cite{fundamental14}. Related to the CS framework, the Restricted Isometry Property (RIP) is generalized in \cite{traonmilin2018stable} and the null space property is generalized in \cite{tran2019class}. However, A. Bastounis et al. establish that the RIP is absent in real-world CS applications in \cite{bastounis2017absence} and, hence, propose an extension thereof, the RIP in levels, in order to explain the success of the method. Related to our work, analogous worst-case and average error bounds for normed spaces and single-valued decoders are established in \cite{plaskota1996noisy}.

\section{Discussion and Outlook}

A fundamental limitation of this work is that the optimal maps achieving the smallest worst-case and average error in \eqref{eq:opt_map_det_ex}, \eqref{eq:optimal_map_inspiration12} and \eqref{eq:optimal_map_inspiration1}, respectively, are generally not analytically computable. Moreover, these optimal maps are defined via embedded optimization problems. While, in principle, such problems could be approximated numerically, doing so may be computationally infeasible for inverse problems involving high-dimensional datasets. 
In the average-case setting, the optimal map admits a measurable selector; however, it remains open to determine conditions under which it also admits a continuous selector or a continuous approximation. This is closely related to identifying structural conditions under which the set-valued map $y \mapsto F_y$ is continuous (in an appropriate sense), which in turn depends on properties of the forward map $F$, the model class $\mathcal{M}_1$, the noise set $\mathcal{E}$, and the underlying metrics. Another limitation is that the worst-case and average kernel sizes are not readily computable. Developing computable and provably convergent approximations to these quantities would enable quantitative accuracy guarantees in applications. In $\S$ \ref{s:examples}, we presented illustrative examples highlighting connections between this framework and other approaches to inverse problems. In particular, Proposition~\ref{prop:Bayes} shows that the optimal map in the average-case setting coincides with the posterior mean estimator in Bayesian inverse problems. A natural direction for future work is to investigate how other Bayesian estimators, such as the maximum a posteriori estimator, relate to the present framework. Finally, extending Proposition \ref{prop:rnspkers} to generalized rNSP conditions such as in \cite{fundamental14} and identifying the corresponding optimal maps in compressed sensing remains an interesting direction.\newline

\noindent The proposed theoretical framework for undetermined and ill-posed inverse problems can be seen as a generalisation of a variety of previous frameworks to set-valued decoders. Additionally, commonly used assumptions, such as the convexity of the set $\mathcal{M}_1$ and a linear forward operator $A$ and additive noise, as well as the condition $(\Mo-\Mo)\cap \mathcal{N}(A)=\{0\}$, which is for example implied by the RIP, are extended or are shown to be unnecessary. Under general assumptions, our work provides relevant accuracy bounds for set- and single-valued decoders and a variational expression for optimal decoders. Due to the generality of the assumptions, these bounds provide a means to bridge the gap between theory and practice. The lower bounds can be used for assessing accuracy and performance of a wide range of models, including ensemble models, diffusion models and sampling-based approaches that have set-valued solutions. Moreover, based on our previous work \cite{SIREV_paper} these bounds can be used in future research to mitigate AI hallucinations in DL-based reconstructions for ill-posed inverse problems.

\section{Proofs}\label{sec:proofmeth}

This section contains the proofs of our results and propositions required to prove these. Additionally, where necessary we include and cite results from prior work that are used in our proofs.

\subsection{Existence and uniqueness of a disintegration of a measure given Assumptions \ref{a:sc}, \ref{a:measur} and \ref{a:Mcompact}}\label{ss:existence}
In this section we prove the following proposition, which ensures that our setting guarantees the existence of a disintegration of the measure $\mu$. 
\begin{proposition}\label{prop:disint}
	Under Assumptions \ref{a:sc}, \ref{a:measur} and \ref{a:Mcompact}, there exists a disintegration $\{\mu^y\}_{y \in \Mt^\E}$ of the measure $\mu$ along $F$. Moreover, such disintegration is essentially unique: if $\{\tilde\mu^y\}_{y \in \Mt^\E}$ is another family satisfying $(i)-(iii)$ in Definition \ref{def:disintegration}, then $\tilde\mu^y=\mu^y$ for almost every $y$.
\end{proposition}

The proof of Proposition \ref{prop:disint} is a special case of what is found in \cite{chang1997conditioning}. Before jumping to the proof, in this section we start by recalling some well known definitions for measures, and state the relevant theorems from \cite{chang1997conditioning}.
 Let $(X,\mathcal{B}(X))$ be the Borel measurable space associated with the topological space $X$. A measure $\mu$ on $(X,\mathcal{B}(X))$ is said to be a \textit{Radon measure} (sometimes also called a regular measure \cite{guide2006infinite}) if
\begin{enumerate}[label=(\roman*)]
    \item $\mu(K)<+\infty$ for each compact $K \in \mathcal{B}(X)$,
    \item $\mu(B) = \inf \{\mu(V): V \in \mathcal{B}(X), V \text{ open, and } B\subset V \}$, for every $B \in \mathcal{B}(X)$ ($\mu$ is outer regular), and 
    \item $\mu(B) = \sup \{\mu(K): K \in \mathcal{B}(X), K \text{ compact, and } K\subset B \}$, for every $B \in \mathcal{B}(X)$ ($\mu$ is tight). 
\end{enumerate}
Moreover, a measure $\mu$ is said to \emph{dominate} a measure $\nu$ if $\mu(B) = 0$ implies $\nu(B)=0$ for every $B\in \mathcal{B}(X)$. Suppose that $X$ can be covered by at most countably many Borel measurable sets $\{B_{i}\}_{i\in I}$, $I\subset \N$, and that $\mu(B_i) < \infty$ for each $i\in I$. Then, $\mu$ is said to be a $\sigma$-finite measure. In particular, every finite measure is $\sigma$-finite. We also have the following results for finite measures, that are used in the proof of Proposition \ref{prop:disint}.

\begin{theorem}[{\cite[Thm.\ 12.7]{guide2006infinite}}]\label{thm:polish}
	A finite measure on a Polish space (i.e. a complete separable metric space) is Radon.
\end{theorem}

In Definition \ref{def:disintegration} one may replace the measure $F_*\mu$ with another measure $\rho$, in which case $\{\mu^y\}_y$ would be called a  $(F,\rho)$ disintegration of $\mu$. We refer the reader to \cite[Def.\ 1]{chang1997conditioning} for the detailed statement.

\begin{theorem}[{\cite[Thm.\ 1]{chang1997conditioning}}]\label{thm:1ch}
	Let $\mu$ be a $\sigma$-finite Radon measure on a metric space $X$ and let $F$ be a measurable map from $(X,\mathcal{B}(X))$ to the measurable space $(Y,\Sigma)$. Let $\rho$ be a $\sigma$-finite measure on $\Sigma$ that dominates the pushforward measure $F_*\mu$. If $\Sigma$ is countably generated and contains all the singleton sets $\{y\}$, then $\mu$ has a $(F,\rho)$-disintegration. The $\mu^y$ measures are
	uniquely determined up to an almost sure equivalence: if $\{\mu^y_*\}$ is another $(F,\rho)$-disintegration then $\rho(\{y\in \mathcal{Y}: \mu^y_*=\mu^y\})=0$.
\end{theorem}

\begin{theorem}[{\cite[Thm.\ 2 (iii)]{chang1997conditioning}}]\label{thm:2ch}
    Let $\mu$ have a $(F,\rho)$-disintegration $\{\mu^y\}$, with $\mu$ and $\rho$ each $\sigma$-finite. Then, the measures $\{\mu^y\}$ are probabilities for $\rho$-almost all $y \in \mathcal{Y}$ if and only if $\rho = F_*\mu$.
\end{theorem}

\begin{proof}[Proof of Proposition \ref{prop:disint}]	
	Let $d:= \max\{d_{\mathcal{X}},d_{\mathcal{Z}}\}$ be one of the standard metrics on the product $\mathcal{X} \times \mathcal{Z}$. The spaces $(\X,d_{\mathcal{X}})$ and $(\Z,d_{\mathcal{Z}})$ are separable, as they are second countable by Assumption \ref{a:sc}\ref{it:as1}, and complete by Assumption \ref{a:sc}\ref{it:as2}. Thus the product space $(\X \times \Z, d)$ is complete and separable, \cite{steen1978counterexamples} (pg. 26, Invariance properties, Table 1). Thus, $(\X \times \Z, d)$
	 is a Polish space. Since by Assumption \ref{a:measur} $\mu$ is a finite measure on the Polish space $\X\times\Z$, Theorem \ref{thm:polish} guarantees that $\mu$ is Radon. Moreover, since by Assumption \ref{a:measur} $\mu$ is finite, it is in particular $\sigma$-finite.
	Taking $\rho=F_*\mu$, clearly $\rho$ dominates $F_*\mu$ trivially, as they are the same measure. 
	By Assumption \ref{a:sc}\ref{it:as1}, $(\Mt^{\E},d_\Y)$ is second countable, hence its topology is countably generated, and so is the corresponding Borel $\sigma$-algebra $\mathcal{B}=\mathcal{B}(\Mt^\E)$. Furthermore, since singletons $\{y\} \subseteq \Mt^\E$ are closed, the Borel $\sigma$-algebra contains all singletons. Lastly, as the mapping $F$ is measurable by Assumption \ref{a:measur} and as by Assumption \ref{a:Mcompact} $\mathcal{M}_1 \times \mathcal{E}$ is a Borel subset of $\mathcal{X} \times \mathcal{Z}$, we extend $F$ trivially to be a measurable function on $\mathcal{X} \times \mathcal{Z}$, for instance by defining it to be constant outside $\mathcal{M}_1 \times \mathcal{E}$. Thus, all the conditions of Theorem \ref{thm:1ch} are satisfied. Hence, there exists a disintegration of the measure $\mu$ along $F$. Such a disintegration is essentially unique in the sense of Theorem \ref{thm:1ch}. Finally, since $\mu$ and $\rho=F_*\mu$ are both finite and in particular $\sigma$-finite, Theorem \ref{thm:2ch} guarantees that $\rho=F_*\mu$ implies that the measures $\mu^y$ are probability measures for $F_*\mu$-almost every $y \in \Mt^{\E}$.
\end{proof}
\subsection{Proof of Theorem \ref{thm:optimal_bounds_sup}}

	\begin{proof}[Proof of Theorem \ref{thm:optimal_bounds_sup}]
        We begin with the proof of (ii). First, let us introduce some notation: for fixed $y \in \Mt^\E$, define the function
		\begin{align*}
			f_y\colon \X \to [0,+\infty), \quad f_y(z) = \sup_{x \in F_y} d_\X(x, z) = d_\X^H(z, F_y)
		\end{align*}
        and notice that $f_y$ appears in the definition of $\Psi$ in \eqref{eq:opt_map_det_ex}. Note that $f_y$ does not attain the value $+\infty$, as by Assumption \ref{a:compact} the feasible set $F_y$ is bounded. That is, we have the relation,
        \begin{equation}\label{eq:psi_def123}
            \Psi(y)= \argmin_{z \in \X} f_y(z).
		\end{equation}
        Furthermore, for each $y \in \Mt^\E$ let 
		\begin{equation*}
            r_y = \sup_{\substack{x \in F_y\\x' \in F_y}} d_\X(x, x') = {\diam}_{d_\X}(F_y).
		\end{equation*}
      Next, we make a claim which, once it is established, we use to prove that $\Psi$ has non-empty, compact values. This ensures that $\Psi$ is well-defined.
        
        \begin{claim}
        For every $y \in \Mt^\E$, we have that,  
            \begin{enumerate}[label=(\Roman*)]
			\item $f_y$ is continuous,
            \item for any $x \in F_y$, we have that
                \[
                    \underset{z \in \X}{\argmin}f_y(z) =\underset{z \in {B}(x,r_y)}{\argmin}f_y(z),\] 
			\item $\Psi(y)$ is closed.
		\end{enumerate}
        \end{claim}
        We proceed to prove the claim and start by considering (I). We consider a general setting, and let $(X,d)$ be a metric space, $A \subseteq X$ be a bounded subset and $g(x)\coloneqq d^H(x,A) = \sup_{a \in A} d(x,a)$. We claim that $|g(x)-g(y)|\leq d(x,y)$ for all $x,y \in X$. To see this, consider $a, x, y \in X$ and note that  
			\begin{align*}
			d(x,a) \leq d(x,y) + d(y,a) \quad \implies \quad \sup_{a \in A}d(x,a) \leq d(x,y) + \sup_{a \in A} d(y,a).
			\end{align*}
        Switching the roles of $x$ and $y$, leads to the desired inequality. It follows that $g$ is continuous. Moreover, $ F_y$ is bounded by Assumption \ref{a:compact}. Thus, letting $(X,d)=(\X, d_\X)$, $A = F_y$, $g=f_y$ above, proves (I).

            To prove (II), we show that no point outside of the closed ball ${B}_{d_\X}(x,r_y)$ can be a minimiser of $f_y$. First, note that $r_y = {\diam}_{d_\X}(\pi_1 (F_y))<\infty$ as $\pi_1 (F_y)$ is bounded. Moreover, by definition, we have that
			\begin{align*}
			r_y = \sup_{x \in F_y} \sup_{x' \in F_y} d_\X(x, x') = \sup_{x \in F_y} f_y(x),
			\end{align*}            
            which implies that $r_y \geq f_y(x)$ for all $x \in F_y$.
            Now, pick an $\hat{x} \in F_y$. If $z\in \X \setminus {B}(\hat{x},r_y)$ is a point outside of the ball, then $d_\X(z,\hat{x}) > r_y$, and 
			\begin{equation*}
               f_y(z) = \sup_{x' \in F_y} d_\X(x',z) \geq d_\X(\hat{x},z) > r_y \geq f_y(\hat{x}).
			\end{equation*}
            It follows that any minimizer of $f_y$ must lie in the ball ${B}(\hat{x},r_y)$. This proves our claim in (II).
			
            Part (III) of the claim follows directly from the continuity of $f_y$. Indeed, let $\alpha_y \coloneqq \min f_y \in [0, +\infty)$. Then, since $\{\alpha_y\}\subset [0,\infty)$ is closed and $f_y$ is continuous, the preimage  $f_y^{-1}(\{\alpha_y\}) = \Psi(y)$ is closed. This proves (III), and concludes our proof of the claim.
		
    Next, we prove the following properties of $\Psi$. These properties  finalize  the proof of statement (ii) of the theorem. 	
        \begin{enumerate}[label=(\alph*)]
            \item \emph{$\Psi$ has non-empty and compact values};
			\item \emph{$\Psi$ is an optimal map};
		\end{enumerate}
	
    We start with the proof of (a).
            Let $y \in \Mt^{\E}$ and consider $x \in F_y$. From (II) and \eqref{eq:psi_def123}, we see that 
                \begin{equation}\label{eq:minargpsi}
                    \Psi(y) = \argmin_{z \in \mathcal{X}}f_y(z) = \argmin_{z \in {B}(x,r_y)} f_y(z).
				\end{equation}
                By Assumption \ref{a:sc}\ref{it:as2}, $d_\X$ satisfies the Heine-Borel property, which  implies that the closed and bounded ball $\mathcal{B}(x,r_y)$ is compact. Moreover, from (I) we know that the objective function $f_y$ is continuous. Thus, it follows from the Extreme Value Theorem that $f_y$ attains its minimum in \eqref{eq:minargpsi}. This implies that $\Psi$ has non-empty values. To see that $\Psi$ has compact values, observe that $\Psi(y)$ is closed by (III) and that $\Psi(y) \subseteq {B}(x,r_y)$, where ${B}(x,r_y)$ is compact. Since a closed subset of a compact set is compact, we conclude that $\Psi(y)$ is compact.

Next, we prove (b). We will show that the minimum worst-case reconstruction error of $\Psi$ equals the optimality constant as given in Definition \ref{eq:mapconstn3}. To this end, let $\varphi\colon \Mt^\E \rightrightarrows \X$ and fix $y \in \Mt^{\E}$. Recall from above that $\Psi$ is non-empty. By construction we have that
\begin{equation*}
    \sup_{x \in F_y} d_\X^H(\Psi (y),x) \leq \sup_{x \in F_y} d_\X(x,z),\quad\text{for all $z \in \varphi(y)$.}
			\end{equation*}
            It follows that
			\begin{align*}
			\sup_{x \in F_y} d_\X^H(\Psi (y),x) \leq  \sup_{x \in F_y} \sup_{z \in \varphi(y)} d_\X(x,z)   =  \sup_{x \in F_y} d_\X^H(x,\varphi(y)).
			\end{align*}
			By taking the supremum with respect to $y\in\Mt^{\E}$ on both sides, we obtain
			\begin{align*}
			\sup_{y \in \Mt^{\E}} \sup_{x \in F_y}d_{\X}^{H} (\Psi(y), x) \leq \sup_{y \in \Mt^{\E}} \sup_{x \in F_y} d_{\X}^{H}(\varphi(y), x).
			\end{align*}
			As $F\colon \Mo\times\E \to \Mt^\E$ is surjective, the previous inequality can be rewritten as
			\begin{align*}
			\sup_{(x,e)\in \Mo \times \E} d_{\X}^{H} (\Psi(F(x,e)), x) \leq \sup_{(x,e)\in \Mo \times \E} d_{\X}^{H} (\varphi(F(x,e)), x).
			\end{align*}
            Now, since $\varphi\colon \Mt^\E \rightrightarrows \X$ was arbitrary, the above inequality holds for any $\varphi\colon \Mt^\E \rightrightarrows \X$.  Taking the infimum over all mappings on the right hand, side we obtain the optimality constant:
			\begin{align*}
			\sup_{(x,e)\in \Mo \times \E} d_{\X}^{H} (\Psi(F(x,e)), x) \leq c_{\mathrm{opt}}^{\text{w}}(F,\Mo,  \E). 
			\end{align*}
			The opposite inequality is trivial, since $\Psi \colon \Mt^\E \rightrightarrows \X$ is one of the reconstruction mappings over which the infimum is taken. Therefore,
			\begin{align*}
			\sup_{(x,e)\in \Mo \times \E} d_{\X}^{H} (\Psi(F(x,e)), x) = c_{\mathrm{opt}}^{\text{w}}(F,\Mo,  \E),
			\end{align*}
			and we conclude that $\Psi$ is an optimal map.
This concludes the proof of statement (ii) in the theorem.

    We proceed with the proof of (i) and start with the lower bound in \eqref{eq:bounds_kersize}.
    Let $\varphi:\Mt^{\E} \rightrightarrows\X$ and $y \in \Mt^{\E}$. Then 
    \begin{align*}
        {\diam}_{d_\X}(F_y) = \sup_{x,x' \in F_y} d_\X(x,x') &\leq
\sup_{x \in F_y} 2d_\X^H(x, \varphi(y)).
    \end{align*}
   Now, taking the supremum over all $y \in \Mt^{\E}$ gives the inequality 
    \begin{align}\label{eq:lower_bound}
        \operatorname{kersize}^{\text{w}}(F,\Mo, \E)\leq \sup_{y \in \Mt^{\E}}  \sup_{x \in F_y} 2d_\X^H(x, \varphi(y)).
	\end{align}
Finally, since $\varphi$ was arbitrary, taking the infimum over all $\varphi:\Mt^\E \rightrightarrows \X$ in \eqref{eq:lower_bound} gives
	\begin{equation*}
		\operatorname{kersize}^{\text{w}}(F,\Mo, \E) \leq 2c_{\mathrm{opt}}^{\text{w}}(F,\Mo,  \E).
	\end{equation*}
which proves the lower bound in \eqref{eq:bounds_kersize}.

For proving the upper bound in \eqref{eq:bounds_kersize}, we will make use of the mapping $\Psi$ in \eqref{eq:opt_map_det_ex}, which we know from statement (ii) is an optimal map. Let $y \in \Mt^{\E}$ and $x' \in F_y$. Then by construction of $\Psi$, we have that     
	\begin{equation*}
        \sup_{x \in F_y} d_\X^H(\Psi (y),x) \leq \sup_{x \in F_y} d_\X(x,x')\leq \diam(F_y).
	\end{equation*}
Taking the supremum over all $y \in \Mt^{\E}$ on both sides above, gives 
	\begin{equation*}
         \sup_{y \in \Mt^{\E}}   \sup_{x \in F_y} d_\X^H(\Psi (y),x) =c_{\mathrm{opt}}^{\text{w}}(F, \mathcal{M}_1, \mathcal{E})  \leq  \operatorname{kersize}^{\text{w}}(F,\Mo, \E),
	\end{equation*}
    where we used the fact that $\Psi$ is an optimal map in the first equality. This establishes statement (i).    

\end{proof}

\subsection{Preliminaries from measure theory}\label{subs:preliminaries}

In order to prove Theorem \ref{thm:optimal_bounds} we need to state and cite some existing results from measure theory and prove a preliminary lemma.

\subsubsection{Carathéodory functions}
Let $(S,\Sigma_{S}), (X,\Sigma_{X})$ and $(Y,\Sigma_{Y})$ be measurable spaces. For functions $f\colon S\times X \to Y$, whose domain is a product of two measurable spaces, several notions of measurability exist. The function $f$ could be \emph{jointly measurable}, i.e., measurable with respect to the product $\sigma$-algebra $\Sigma_{S}\otimes \Sigma_{X}$. For fixed $s \in S$ or $x \in X$, the functions $f^s = f(s,\cdot) \colon X \to Y$ and $f^x = f(\cdot,x) \colon S \to Y$, could be measurable with respect to $\Sigma_{X}$ and $\Sigma_{S}$, respectively. A function $f$ which is measurable in one variable for each fixed $s \in S$ and for each fixed $x \in X$ is said to be \emph{separately measurable}. In general, joint measurability implies separate measurability, but the converse is not true \cite[p. 152]{guide2006infinite}.         

A class of functions which are jointly measurable in many important cases are Carathéodory functions.
\begin{definition}[{\cite[Def.\ 4.50]{guide2006infinite}}]
    Let $(S,\Sigma)$ be a measurable space, and let $X$ and $Y$ be topological spaces. A function $ f\colon S\times X \to Y$ is a \emph{Carathéodory function} if:
\begin{enumerate}[label=(\roman*)]
    \item for each $x \in X$, the function $f^{x} = f(\cdot,x)\colon S\to Y$ is $(\Sigma, \mathcal{B}(Y))$-measurable, and
    \item for each $s \in S$, the function $f^{s}=f(s,\cdot) \colon X\to Y$ is continuous. 
\end{enumerate}
\end{definition}
In particular, if $X$ is a separable metric space and $Y$ is a metric space, then every Carathéodory function $f\colon S\times X \to Y$ is jointly measurable \cite[Lem.\ 4.51]{guide2006infinite}. It is also straightforward to see that if $f$ is continuous in both arguments separately, then $f$ is a  Carathéodory function.

\subsubsection{Measurability of set-valued functions}

For a single-valued function $f \colon Y\to X$ the inverse image of a set $A \subset X$ is $f^{-1}(A) \coloneqq \{y \in Y \colon f(y) \in A\}$. For a set-valued function $\varphi\colon Y \rightrightarrows X$ there are different ways to generalize the concept of an inverse image: Given a set $A \subset X$, we say that the \emph{upper inverse} of $A$ is $\varphi^{\mathrm{u}}(A) \coloneqq \{y \in Y :\varphi(y) \subset A\}$ and we say that the \emph{lower inverse} of $A$ is
\[\varphi^{\ell}(A) \coloneqq\{y \in Y :\varphi(y) \cap A\neq \emptyset\}.\]
Naturally, the notion of continuity of set-valued mappings between topological spaces depends on the definition of inverse. Interested readers are referred to \cite[Ch.\ 17]{guide2006infinite}. In this work we shall be most concerned with the lower inverse of a set-valued mapping, since this notion of inverse allows us to pick a measurable selector. 
Next, we define two different notions of measurability for set-valued mappings based on the concept of lower inverse.
\begin{definition}[{\cite[Def.\ 18.1]{guide2006infinite}}]\label{def:measurability}
	 Let $(Y, \Sigma)$ be a measurable space and $X$ a topological space. 
	We say that a set-valued mapping $\varphi: Y \rightrightarrows X$ is:
    \begin{itemize}
        \item 
	 	\emph{weakly measurable}, if 
	 	$\varphi^{\ell}(V) \in \Sigma$ for each open subset $V$ of $X$,
        \item \emph{measurable}, if $\varphi^{\ell}(A)\in \Sigma$ for each closed subset $A$ of $X$. 
    \end{itemize}
\end{definition}
\noindent Importantly, one needs to remark here, that the definition of weakly measurable for set-valued mappings in \cite[Def.\ 18.1]{guide2006infinite} is equivalent to that of measurable set-valued mappings \cite[Def.\ 8.1.1. (i),(ii)]{aubin2009set}. Additionally, we have that in some cases measurability implies weak measurability for set-valued mappings.
\begin{lemma}[{\cite[Lem.\ 18.2]{guide2006infinite}}]\label{lem:equiv} Let $\varphi : Y \rightrightarrows X$ be a set-valued mapping from a measurable space $(Y,\Sigma)$ into a metric space. Then the following hold:
\begin{enumerate}
    \item If $\varphi$ is measurable, then it is also weakly measurable.
    \item If $\varphi$ is compact-valued and weakly measurable, then it is measurable.
\end{enumerate}
\end{lemma}
\noindent Note that if the topology of $X$ above is induced by a metric, then every set-valued measurable mapping is also weakly measurable \cite[Lem.\ 18.2]{guide2006infinite}. Our main tool for proving Theorem \ref{thm:optimal_bounds} is the following result.

\begin{theorem}[{Measurable Maximum Theorem, \cite[Thm.\ 18.19]{guide2006infinite}}]\label{thm:mmt}
	Let $X$ be a separable metric space and $(S,\Sigma)$ a measurable space. Let $\varphi\colon S \rightrightarrows X$ be weakly measurable with non-empty compact values, and suppose $f\colon S \times X \to \R$ is a Carathéodory function. Define the value function $m\colon S \to \R$ by
	\begin{equation*}
		m(s) = \max_{x \in \varphi(s)} \ f(s,x),
	\end{equation*}
	and the correspondence $\Phi\colon S \rightrightarrows X$ of maximisers by
	\begin{equation*}
		\Phi(s) = \{x \in \varphi(s) \ : \ f(s,x) = m(s)\}=\operatorname*{argmax}_{x \in \varphi(s)} f(s,x).
	\end{equation*}
	Then $m$ is measurable, $\Phi$ has non-empty and compact values, and $\Phi$ is measurable and admits a measurable selector.
\end{theorem}

\noindent Note that for real-valued functions $\min_{x \in X} f(x) = \max_{x \in X} (-f(x))$ and the Measurable Maximum Theorem can be applied to minimisation problems as well. The Measurable Maximum Theorem will be applied in two ways. The first can be found in Proposition \ref{prop:C}, which gives sufficient conditions for a map $\phi\colon \Mt^\E \rightrightarrows \X$ to be in the set $\mathcal{C}$ given by \eqref{eq:C_def}. The second use of the theorem above is found in the proof of part (ii) and (iii) of Theorem \ref{thm:optimal_bounds}.

\begin{proposition}\label{prop:C}
	If $\phi: \Mt^\E \rightrightarrows \X$ is measurable and has non-empty compact values, then $\phi \in \mathcal{C}.$
\end{proposition}
\begin{proof}
	By definition of $\mathcal{C}$, we need to prove that the function
	\begin{align*}
	r_\phi: \Mo \times \E \to \R, \quad r_\phi(x,e) = d_{\X}^H (x, \phi(F(x,e))) = \sup_{z \in \phi(F(x,e))} d_\X(x,z)
	\end{align*}
	is measurable. Denote $S := \Mo \times \E$, $X := \X$, and 
	\begin{align*}
	&\varphi := \phi \circ F: \Mo\times \E \rightrightarrows \X, & \varphi(x,e) = \phi(F(x,e)),  \\
	&f := d_\X \circ (\pi_1, \text{id}_{\X}): (\Mo \times \E) \times \X \to \R, &f((x,e), z) = d_\X(x,z).
	\end{align*} 
    We now show that the set-valued mapping $\varphi := \phi \circ F: \Mo\times \E \rightrightarrows \X$ is weakly measurable. Let $V \subset \mathcal{X}$ be an open set. Then, by the definition \cite[$\S$\ 17.1]{guide2006infinite} of the lower inverse $\varphi$ is given by
\begin{align*}
\varphi^\ell(V)
= \{(x,e) \in \mathcal{M}_1 \times \mathcal{E} : \varphi(x,e) \cap V \neq \emptyset \} 
= \{(x,e) \in \mathcal{M}_1 \times \mathcal{E} : \phi(F(x,e)) \cap V \neq \emptyset \}.
\end{align*}
Observe that this set can be written as $\varphi^\ell(V) = F^{-1}(\phi^\ell(V)),$
where $\phi^\ell(V) = \{y \in \mathcal{M}_2^{\mathcal{E}} : \phi(y) \cap V \neq \emptyset\}$ denotes the lower inverse of $V$ under $\phi$. Since $\phi$ is a set-valued mapping that is measurable, by Lemma~\ref{lem:equiv} it follows that $\phi$ is also weakly measurable. As $\phi: \Mt^\E \rightrightarrows \X$ is weakly measurable and as $V \subset \mathcal{X}$ is an open set, we have $\phi^\ell(V) \in \mathcal{B}(\mathcal{M}_2^{\mathcal{E}})$. Moreover, $F\colon (\Mo\times\E, \mathcal{B}(\Mo\times\E)) \to (\Mt^\E, \mathcal{B}(\Mt^\E))$ is Borel measurable by Assumption \ref{a:measur}. Therefore, the preimage of the Borel set $\phi^\ell(V) \in \mathcal{B}(\mathcal{M}_2^{\mathcal{E}})$ under $F$ is Borel measurable
\[
\varphi^\ell(V) = F^{-1}(\phi^\ell(V)) \in \mathcal{B}(\mathcal{M}_1 \times \mathcal{E}).
\]
Hence, the set-valued mapping $\varphi: \Mo\times \E \rightrightarrows \X$ is weakly measurable. Moreover, $\varphi: \Mo\times \E \rightrightarrows \X$ has non-empty compact values because $\phi: \Mt^\E \rightrightarrows \X$ has non-empty compact values by assumption. By Assumption \ref{a:measur} $(\Mo\times \E, \mathcal{B}(\Mo\times \E), \mu)$ is a $\sigma-$finite measure space. The space $(\X,d_{\mathcal{X}})$ is separable, as it is second countable by Assumption \ref{a:sc}\ref{it:as1}, and together with Assumption \ref{a:sc}\ref{it:as2} it is also complete. Finally, the function $f$ is continuous in both arguments and hence Carathéodory. Thus, the assumptions in Theorem \ref{thm:mmt} are satisfied, and its application implies that the value function $m: \Mo \times \E\to \R$
	\begin{align*}
	m(x,e) = \max_{z \in \phi(F(x,e))} d_\X(x,z)
	\end{align*}
	is measurable. It is immediate to notice that $m = r_\phi$, so we conclude that $r_\phi$ is measurable. Hence we have proven that $\phi \in \mathcal{C}$.
\end{proof}

\subsubsection{Disintegration of measures}

Next we prove a useful proposition that will be applied in the proof of Theorem \ref{thm:optimal_bounds}. For completeness, we start by recalling that every time we write an expression of the form
	\begin{align*}
\essup_{(x,e) \in \Mo \times \E} f(x,e), \quad \quad	\essup_{(x,e) \in F^{-1}(y)} f(x,e),  \quad\quad	\essup_{y \in \Mt^\E} f(y).
	\end{align*}
	Whenever it is clear from context, it is implicit that we are taking the essential supremum with respect to $\mu$ on $\Mo \times \E$, with respect to $\mu^y$ on $F_y$, and with respect to $F_*\mu$ on $\Mt^\E$. We also recall that $\pi_1\colon\X\times\Z\to\X$ refers to the projection on the first component $\pi_1(x,e)=x$.

\begin{proposition}\label{prop:disintegration_properties}
	Given Assumptions \ref{a:sc}, \ref{a:measur} and \ref{a:Mcompact}, we have the following:
	\begin{enumerate}[label=(\roman*)]
		\item For $A \in \mathcal{B} (\Mo\times\E)$, $\mu(A) =0$ if and only if $\mu^y(A) = 0 $
		for almost every $y \in \Mt^{\E}$.

    \item Let $f: \Mo\times \E\to [0,+\infty)$ be Borel measurable. Then the function
		\begin{align*}
		m\colon \Mt^{\E} \to \R, \quad m(y) = \essup_{(x,e) \in F^{-1}(y)} f(x,e)
		\end{align*}
		is Borel measurable.
    \item Let $f: \Mo\times \E \to [0,+\infty)$ be Borel measurable. Then
		\begin{align*}
		\essup_{(x,e) \in \Mo\times\E} f(x,e) = \essup_{y \in \Mt^\E} \essup_{(x,e) \in F^{-1}(y)} f(x,e).
		\end{align*}
    \item Let $g: \Mo \to [0,+\infty)$ be Borel measurable. Then
		\begin{align*}
		\int_{F^{-1}(y)} g(\pi_1(x,e))d\mu^y(x,e) = \int_{F_y} g(x) d\big(\pi_{1*}\mu^y\big)(x),
		\end{align*}
		and
	\begin{align*}
		\mu^y-\essup_{(x,e) \in F^{-1}(y)} g(\pi_1(x,e)) = \pi_{1*}\mu^y-\essup_{x\in F_y} g(x),
	\end{align*}
    where $\mu^y-\essup$, respectively $\pi_{1*}\mu^y-\essup$, denote the essential supremum with respect to the disintegration of measure $\mu^y$ and, respectively, the pushforward of the disintegration of measure $\pi_{1*}\mu^y$.
	\end{enumerate} 
\end{proposition}
\begin{proof}
	We start with $(i)$.
	 Let $1_A$ denote the indicator function on $A \in \mathcal{B}(\mathcal{M}_1\times \mathcal{E})$. Suppose that $\mu(A) = 0$, then by (iii) in Definition \ref{def:disintegration}, we have that 
		\begin{equation*}
	0 = \mu(A) = \int_{\Mo \times \E} 1_A \ d\mu = \int_{\Mt^{\E}} \Bigg( \int_{F^{-1}(y)} 1_A \ d\mu^y \Bigg)\ d(F_*\mu)(y).
		\end{equation*}
		The integral of a positive function with respect to a positive measure is zero if and only if the integrand function is zero almost-everywhere, so the previous equality is equivalent to 
		\begin{equation*}
		 \int_{F^{-1}(y)} 1_A \ d\mu^y = \mu^{y}(A) = 0 \quad  \text{for }\text{ almost every } y \in \Mt^{\E}.
		\end{equation*}
The reverse argument can be made with the same steps. Let us now prove $(ii)$. Since the Borel $\sigma$-algebra on $\mathbb{R}$ is generated by sets of form $(a, +\infty)$, we proceed to show that $m^{-1}((a, +\infty))$ is Borel measurable for any $a \in \R$. We have 
		\begin{align*}
		m^{-1}((a,+\infty)) & = \{y \in \Mt^{\E}: m(y) > a\} \\
		& = \{y \in \Mt^{\E}: \mu^y(\{(x,e) \in \Mo\times\E: f(x,e) > a\})>0\} \\
		& = \{y \in \Mt^{\E}: \int_{\Mo\times\E} 1_{f^{-1}((a,\infty))} \ d\mu^y>0\}  \\
		& = \{y \in \Mt^{\E}: h_a(y) >0\}  = h_a^{-1}((0,+\infty)),
		\end{align*}
		where $h_a(y) = \int_{\Mo\times\E} 1_{f^{-1}((a,\infty))} \ d\mu^y$. In particular, since $f$ is Borel measurable, $f^{-1}((a,\infty))$ is a measurable set and hence $1_{f^{-1}((a,\infty))}$ is a measurable function on $(\Mo\times\E,\mathcal{B}(\mathcal{M}_1\times\mathcal{E}))$. Now, from (ii) in Definition \ref{def:disintegration} we know that the function $y \mapsto h_a(y) = \int_{\Mo\times\E} 1_{f^{-1}((a,\infty))} \ d\mu^y$ is measurable and so $h_a^{-1}((0,+\infty)) = m^{-1}((a,+\infty))$ is a Borel measurable subset of $\mathcal{M}_2^{\mathcal{E}}$. As $a \in \R$ was arbitrary, this proves that $m$ is Borel-measurable. This concludes $(ii)$.
Next we consider $(iii)$, and let $m$ be as in $(ii)$. We start by showing that
		\begin{align}\label{eq:geq432}
	\essup_{(x,e) \in \Mo\times\E} f(x,e) \geq \essup_{y \in \Mt^{\E}} m(y).
	\end{align}
Let $K= \essup_{(x,e) \in \Mo\times\E} f(x,e)$. By definition we have $\mu(\{(x,e) \in \Mo\times\E: f(x,e) > K \})=0$. Hence, by Proposition \ref{prop:disintegration_properties}, $(i)$, and the fact that $\mu^y$ is concentrated on $F^{-1}(y)$, we have 
\begin{align*}
	0&=\mu^y(\{(x,e) \in \Mo\times\E: f(x,e) > K\})
	  =\mu^y(\{(x,e) \in F^{-1}(y): f(x,e) > K \})
\end{align*} for almost every $y \in \Mt^{\E}$. 
So $m(y) = \essup_{(x,e) \in F^{-1}(y)} f(x,e) \leq K = \essup_{\Mo\times\E} f$ for almost every $y \in \Mt^{\E}$. This proves \eqref{eq:geq432}. Next, we consider the reverse inequality. By definition of $\essup_{\Mt^{\E}} m$, we have \begin{align*}
	 m(y) =\essup_{(x,e) \in F^{-1}(y)} f(x,e) \leq \essup_{\Mt^{\E}} m
	\end{align*}
	for almost every $y \in \Mt^{\E}$. Hence, as $\mu^y$ is concentrated on $F^{-1}(y)$
	\begin{align*}
	\mu^y(\{(x,e) \in F^{-1}(y): f(x,e) > \essup_{\Mt^{\E}} m \})=\mu^y(\{(x,e) \in \Mo\times\E: f(x,e) > \essup_{\Mt^{\E}} m \})=0
	\end{align*} 
	for almost every $y \in \Mt^{\E}$. By Proposition \ref{prop:disintegration_properties}, $(i)$, this is equivalent to $\mu(\{(x,e) \in \Mo\times\E: f(x,e) > \essup_{\Mt^{\E}} m \})=0$, which implies $\essup_{(x,e) \in \Mo\times\E} f(x,e) \leq \essup_{y \in \Mt^{\E}} m(y)$, as desired. This concludes the proof of $(iii)$.
	
	 Finally, let us prove $(iv)$. Equality between integrals
	\begin{align*}
	\int_{F^{-1}(y)} (g \circ \pi_1) \, d\mu^y = \int_{F_y} g d\pi_{1*}\mu^y
	\end{align*}
    comes directly from the definition of pushforward measure, and can be found in detail \cite[Theorem 13.46]{guide2006infinite}. To see the equality between essential suprema, take $M \in [0,+\infty)$, then note that
	\begin{align*}
	(\pi_{1*}\mu^y)\big(g^{-1}(M, +\infty)\big) = \mu^y(\pi_1^{-1}(g^{-1}(M, +\infty))) = \mu^y((g\circ \pi_1)^{-1}(M, +\infty)).
	\end{align*}
	Taking $M = \essup_{F^{-1}(y)} (g\circ \pi_1)$, then the definition of essential supremum gives $\mu^y((g\circ \pi_1)^{-1}(M, +\infty))=0$, which by the previous chain of equalities implies that $0 =(\pi_{1*}\mu^y)(\{x \in F_y: g(x) > \essup_{F^{-1}(y)} (g\circ \pi_1)\})$, so that $g(x) \leq \essup_{F^{-1}(y)} (g\circ \pi_1)$ for $\pi_{1*}\mu^y$-almost every $x \in F_y$. This gives $\essup_{x \in F_y} g(x) \leq \essup_{F^{-1}(y)} (g\circ \pi_1)$, 
	By taking instead $M = \essup_{x \in F_y} g(x)$, the same reasoning proves the reverse inequality between essential suprema. Hence equality is proven, concluding the proof of $(iv)$.
\end{proof}

\subsection{Proof of Theorem \ref{thm:optimal_bounds}}

Having established the above preliminaries, we now proceed to proving Theorem \ref{thm:optimal_bounds}. In order to do so, we will split the Theorem up into two results, Proposition \ref{prop:optimal_bounds1} and Proposition \ref{prop:well-defined} and prove these. These results combined provide the proof of Theorem \ref{thm:optimal_bounds}.

\begin{proposition}[Lower bound of part $(i)$, Theorem \ref{thm:optimal_bounds}]\label{prop:optimal_bounds1}
	Given the assumptions in Theorem \ref{thm:optimal_bounds}, then the following holds for every $p \in [1,\infty]$:
	\begin{align*}
		\operatorname{kersize}^{\text{a}}(F,\Mo, \E, p) \leq 2c_{\mathrm{opt}}^{\text{a}}(F,\Mo,  \E,  p).
	\end{align*}
\end{proposition}

\begin{proof}[Proof of Proposition \ref{prop:optimal_bounds1}]
	Let $\varphi:\Mt^{\E} \rightrightarrows\X$ be an arbitrary reconstruction mapping. Fix  $y \in \Mt^{\E}$ and consider $(x,e),(x',e') \in \Mo\times\E$ such that $F(x,e) = F(x',e')= y$. Then, by the triangle inequality, we deduce
\begin{equation}\label{eq:triangle1}
	d_\X(x,x') \leq d_\X^H(x, \varphi(F(x,e))) + d_\X^H(x',\varphi(F(x',e'))).
\end{equation} 
Let us now distinguish the two cases where $p = \infty$ and $p\in [1, \infty)$. The structure of the proof will be very similar in the two cases, with the main difference that the case $p = \infty$ involves essential suprema, while the case $p \in[1,+ \infty)$ involves integrals. We will fully prove the case $p\in [1, \infty)$ and provide a sketch of the proof of the case $p=\infty$, as it is virtually identical. In both cases, however, equation \eqref{eq:triangle1} will play a crucial role. To ensure that the essential suprema and integrals are well-defined we assume that $\varphi \in \mathcal{C}$.

\emph{Case  $p\in [1, \infty)$}; Here we use that for $a \in [0,\infty)$ the map $a \mapsto a^p$ is convex and by the definition of convexity, we have that for $a,b \in [0,\infty)$ that $( \tfrac{1}{2} a + \tfrac{1}{2} b)^{p} \leq \tfrac{1}{2} a^{p} + \tfrac{1}{2} b^{p}$. Replacing $a,b$ with $2a,2b$ yields $(a+b)^p \leq 2^{p-1}(a^p + b^p)$ for $a,b\geq 0$. Now integrating \eqref{eq:triangle1} twice with respect to $\mu^y$ and using that $(a+b)^p \leq 2^{p-1}(a^p + b^p)$ for $a,b\geq 0$, we obtain
\begin{align*}
	&\int_{F^{-1}(y)} \int_{F^{-1}(y)} d_\X(x,x')^p \, d\mu^y(x,e) \, d\mu^y(x',e')   \\
	\leq&\int_{F^{-1}(y)} \int_{F^{-1}(y)} \Big(d_\X^H(x, \varphi(F(x,e))) + d_\X^H(x',\varphi(F(x',e'))) \Big)^p \ d\mu^y(x,e) \, d\mu^y(x',e')  \\
	=&  2^p\int_{F^{-1}(y)} d_\X^H(x, \varphi(F(x,e)))^p \ d\mu^y(x,e)
\end{align*}
where in the last step we also used the fact that $\mu^y$ is a probability measure and the integrals are well-defined as $\varphi \in \mathcal{C}$. Now, integrating the above inequality with respect to $F_*\mu$ on $\Mt^{\E}$, by (iii) in Definition \ref{def:disintegration}, and raising to the power $\frac{1}{p}$ gives that
\begin{align*}
	\operatorname{kersize}^{\text{a}}(F,\Mo, \E, p) & = \Bigg(\int_{\Mt^{\E}} \int_{F_y} \int_{F_y} d_\X(x,x')^p \ d\mu^y(x,e) \ d\mu^y(x',e') \, d(F_*\mu)(y) \Bigg)^\frac{1}{p} \\
	&\leq \Bigg( 2^p \int_{\Mt^{\E}} \int_{F_y} d_\X^H(x, \varphi(F(x,e)))^p \ \ d\mu^y(x,e) \ d(F_*\mu)(y)\Bigg)^\frac{1}{p} \\
	& = 2  \Bigg( \int_{\Mo\times\E} d_\X^H(x, \varphi(F(x,e)))^p \ d\mu(x,e)  \Bigg)^\frac{1}{p}.
\end{align*} 
Since $\varphi  \in \mathcal{C}$ was arbitrary, by taking the infimum over $\varphi  \in \mathcal{C}$ we obtain:
\begin{align*}
	\operatorname{kersize}^{\text{a}}(F,\Mo, \E, p) &\leq 2\inf_{\varphi \in \mathcal{C}}\Bigg( \int_{\Mo\times\E} d_\X^H(x, \varphi(F(x,e)))^p \ d\mu(x,e)  \Bigg)^\frac{1}{p} \\
	& = 2c_{\mathrm{opt}}^{\text{a}}(F,\Mo,  \E, p).
\end{align*}
The proposition is therefore also proven in the case $p \in [1,+\infty)$.

\emph{Case $p = \infty$}; the proof follows the basic structure of the case $p \in [1,\infty)$. In \eqref{eq:triangle1} instead of integrating we take the essential supremum respect to the measure $\mu^y$. Then in the next step, instead of integrating, we take essential supremum in $y$ with respect to the measure $F_*\mu$ on $\Mt^{\E}$ and apply Proposition \ref{prop:disintegration_properties}. Finally, as $\varphi:\Mt^{\E} \rightrightarrows\X$ was arbitrary, taking the infimum over $\varphi  \in \mathcal{C}$ concludes the proof.

\end{proof}

\begin{proposition}[Part $(ii)$ and $(iii)$ and upper bound in (i) of Theorem \ref{thm:optimal_bounds}]\label{prop:well-defined}
	The following holds for every $p \in [1,\infty]$:
    \begin{equation*}
		c_{\mathrm{opt}}^{\text{a}}(F,\Mo,  \E,  p) \leq \operatorname{kersize}^{\text{a}}(F,\Mo, \E, p),
	\end{equation*}
	and  the map $\Psi\colon \Mt^{\E} \rightrightarrows \X$ given by,
	\begin{align}\label{eq:optimal_map_inspiration2}
		\Psi(y) &= \argmin_{z \in \X}\essup_{(x,e) \in F^{-1}(y)} d_\X(x,z) & (p = \infty) \\ \label{eq:optimal_map_inspiration3}
		\Psi(y) &= \argmin_{z \in \X} \int_{F^{-1}(y)} d_\X(x,z)^p \ d\mu^y(x,e)  & (p \in [1, \infty))
	\end{align}
	is an optimal map with average error of order $p$. Moreover, $\Psi$ has non-empty compact values, is measurable and it admits a measurable selector.
\end{proposition}

\begin{proof}[Proof of Proposition \ref{prop:well-defined}]	
	We distinguish the cases $p = \infty$ and $p \in [1,\infty)$.
	In both cases, the structure of the proof consists in proving the following steps:

    \begin{enumerate}[label=(\alph*)]
		\item \emph{$\Psi$, defined either by \eqref{eq:optimal_map_inspiration2} or \eqref{eq:optimal_map_inspiration3}, has non-empty values;}
		\item \emph{$\Psi$ is measurable, has compact values and  admits a measurable selector;}
		\item \emph{$\Psi \in \mathcal{C} = \{\varphi: \mathcal{M}_2^\E \rightrightarrows \X: (x,e) \mapsto d_\X^H(x,\varphi(F(x,e)) \text{ is measurable}\}$};
		\item \emph{$\Psi$ is an optimal map with average error of order $p$};
		\item \emph{Upper bound: $c_{\mathrm{opt}}^{\text{a}}(F,\Mo,  \E, p) \leq  \operatorname{kersize}^{\text{a}}(F,\Mo,\E,p)$ with $p=\infty$ or $p \in [1,+\infty)$}.
	\end{enumerate}

\noindent \emph{First we consider the case $p \in [1,\infty)$.}
Again, let us first introduce some notation, similar to above. Fix $y \in \Mt^{\E}$. Define $f_y: \X \to [0,\infty)$,
	\begin{equation*}
	f_y(z) = \int_{F^{-1}(y)} d_\X(x, z)^p  \ d\mu^y(x,e)
	\end{equation*}
	for $z \in \X$. Define also 
	\begin{align*}
	r_y &= \essup_{\substack{(x,e) \in F^{-1}(y)\\(x',e') \in F^{-1}(y)}} d_\X(x, x'), \\
	E_y^{(x,e)} &= \{(x',e') \in F^{-1}(y): d_\X(x,x') > r_y\} \quad \text{ for } (x,e) \in F^{-1}(y),\\
	G_y &= \{(x,e) \in F^{-1}(y): \mu^y(E_{y}^{(x,e)}) = 0 \}.
	\end{align*}
\noindent Note that $f_y$ is Borel measurable. The function $((x,e),z) \mapsto 1_{F^{-1}(y)}(x,e)\, d_{\X}(x,z)$ is non-negative and Borel measurable. By the Fubini-Tonelli \cite[Theorem~2.37]{folland1999real} $f_y$, as $z \mapsto \int_{F^{-1}(y)} d_\X(x, z)^p  \ d\mu^y(x,e)$, is Borel measurable. Additionally, by a similar line of reasoning, one can show that the function $y \mapsto r_y$ is Borel measurable. Additionally, as the metric space $(\mathcal{M}_1,d_{\mathcal{X}})$ is compact by Assumption \ref{a:Mcompact}, the set $\mathcal{M}_1$ is bounded. Thus, as $F^{-1}(y) \subset \mathcal{M}_1\times\mathcal{E}$, we have that 
\begin{align*}
r_y \leq \sup_{\substack{(x,e) \in F^{-1}(y)\\(x',e') \in F^{-1}(y)}} d_\X(x, x') \leq \diam(\mathcal{M}_1)<\infty.
\end{align*}\newline
\noindent We now show that the set $G_y$ is non-empty and $\mu^y(G_y)=1$. Fix $y \in \mathcal{M}^{\mathcal{E}}$. By definition of the essential supremum, it follows that
       \begin{align*}
        (\mu^y \otimes \mu^y)\Big(\big\{((x,e),(x',e')) \in F^{-1}(y)\times F^{-1}(y) : d_{\mathcal{X}}(x,x') > r_y \big\}\Big)=0.
          \end{align*}
        Then, by Fubini-Tonelli \cite[Theorem~2.37]{folland1999real},
       \begin{align*}
        0 &= (\mu^y \otimes \mu^y)\left(\big\{((x,e),(x',e')) \in F^{-1}(y)\times F^{-1}(y) : d_{\mathcal{X}}(x,x') > r_y \big\}\right)\\
        &= \int_{F^{-1}(y)} \mu^y\big(E_y^{(x,e)}\big)\, d\mu^y(x,e).
          \end{align*}
        Since the integrand is nonnegative, we conclude that
        \begin{align*}
        \mu^y\big(E_y^{(x,e)}\big)=0 \quad \text{for } \mu^y\text{-a.e. } (x,e)\in F^{-1}(y).
      \end{align*}
        Recall the definition $G_y = \{(x,e) \in F^{-1}(y) : \mu^y(E_y^{(x,e)}) = 0\}$. As the disintegration $\mu^y$ is a probability measure that is concentrated on $F^{-1}(y)$, Definition \ref{def:disintegration} (i), we have that $\mu^y(G_y)=1$ and, in particular, $G_y \neq \emptyset$.\newline

	\textbf{Claim.} We claim that the following holds: for every $y \in \Mt^\E$
		\begin{enumerate}
		\item[(I)] $f_y$ is continuous,
		\item[(II)] $\underset{\X}{\argmin}f_y =\underset{B(x,2r_y)}{\argmin}f_y$ for $\mu^y$-almost every $(x,e)\in F^{-1}(y)$.
	\end{enumerate}

We proceed to prove the claim and start by considering (I). We consider a general setting, and let $(X,d)$ be a metric space equipped with a probability measure $\nu$ concentrated on a bounded subset $A\subseteq X$, and define 
\[
g(x) \coloneqq \Big( \int_A d(x,a)^p \, d\nu(a)\Big)^\frac{1}{p}=\|d(x,\cdot)\|_{L^p(X,\nu)} = \|d(x,\cdot)\|_p.
\]
 We claim that $|g(x)-g(z)|\leq d(x,z)$ for all $x,y \in X$. To see this, consider $a, x, z \in X$ and note that by the triangle inequality $d(x,a) \leq d(x,z) + d(z,a)$. Taking $L^p(X,\nu)$-norms in the variable $a$ and applying Minkowski's inequality gives
\begin{align*}
 g(x) = \|d(x,\cdot)\|_p \leq \|d(x,z)+d(z,\cdot)\|_p \leq \|d(x,z)\|_p+\|d(z,\cdot)\|_p = d(x,z) + g(z),
\end{align*}
where in the last passage we used that $\nu$ is a probability measure.
Switching the roles of $x$ and $z$, leads to the desired inequality. It follows that $g$ is continuous. By assumption \ref{a:Mcompact}, since $\Mo$ is compact and $F_y\subseteq \Mo$, then $F_y$ is bounded. Thus, letting $(X,d) = (\X,d_\X)$, $\nu = \pi_{1*}\mu^y$, $A = F_y$ above, and recalling Proposition \ref{prop:disintegration_properties}, (iv), then $g^p=f_y$, which proves (I).

To prove (II) we will show that for $\mu^y$-almost every $(x,e) \in F^{-1}(y)$ we have
	\begin{align}\label{eq:noiselessargmin23ee}
	\Psi(y) = \argmin_{z\in \mathcal{X}} \int_{F^{-1}(y)}d_\X(x',z)^p\ d\mu^y(x',e') = \argmin_{z\in B_{d_\X}(x,2r_y)} f_y(z).
	\end{align}
	Fix $(x,e) \in G_y$. If $z\in \X \setminus B_{d_\X}(x,2r_y)$, then for $\mu^y$-almost every $(x',e')\in F^{-1}(y)$,
	\begin{equation*}
	d_\X(z,x') \geq d_\X(z,x)-d_\X(x,x') > 2r_y - r_y = r_y.
	\end{equation*}
	Thus, the previous inequality holds for any $z\in \X \setminus B_{d_\X}(x,2r_y)$. Thus,
	\begin{align*}
	f_y(z) &= \int_{F^{-1}(y)}d_\X(z,x')^p\ d\mu^y(x',e')    \\
	&> \int_{F^{-1}(y)} r_y^p \, d\mu^y(x',e') = r_y^p.
	\end{align*}
	On the other hand, 
	\begin{align*}
	f_y(x) &=  \int_{F^{-1}(y)}d_\X(x,x')^p\ d\mu^y(x',e')  \\
	& \leq  \int_{F^{-1}(y)} r_y^p \ d\mu^y(x',e') = r_y^p.
	\end{align*}
	Therefore $f_y(z) > f_y(x)$ whenever $z \notin B(x,2r_y)$, which implies that points $z$ outside of such ball cannot be minimisers, hence proving \eqref{eq:noiselessargmin23ee}. This concludes part (II) of the claim.
	
	With the claim, we can prove the required (a)-(e) properties of $\Psi$.
	First, let us prove (a), namely that $\Psi$ in \eqref{eq:optimal_map_inspiration3} has non-empty values. By the Heine-Borel property of the metric $d_\X$, the set $B_{d_\X}(x,2r_y) \subset \mathcal{X}$ is compact, since it is a closed ball with respect to the metric $d_\X$. Hence, the function $f_y$ is continuous by (I) and its minimisers are by (II) restricted to the compact set $B_{d_\X}(x,2r_y)$. Therefore, the minimum in \eqref{eq:optimal_map_inspiration3} is attained by the Extreme Value Theorem. This shows that the argmin is non-empty, and hence that $\Psi$ has non-empty values on $\mathcal{M}_2^\E$.

	 \noindent We now proceed to prove (b), namely that $\Psi$ is measurable and has compact values. Recall that above we have shown that $\diam(\mathcal{M}_1)<\infty$. We will apply the Maximum Measurable Theorem, \ref{thm:mmt}, with 
		$S = \Mt^{\E}$, 
		$X = \X$, and
		\begin{align*}
		\varphi &: \Mt^{\E} \rightrightarrows \X, \qquad \varphi(y) = B_{d_\X}(\Mo,2\diam(\Mo)),  \\
		f &:\Mt^{\E} \times \X \to \R, \ f(y,z) = f_y(z) = \int_{F^{-1}(y)} d_\X(x ,z)^p \ d\mu^y(x,e).
		\end{align*}
		In order to apply the theorem, we need to verify that the assumptions are satisfied. We start by proving that $\varphi$ is weakly-measurable with non-empty compact values. 
		Firstly, it is clear that, since $\varphi$ is constant, then $\varphi$ is weakly measurable and has non-empty values. Moreover, the only value $\varphi$ takes is   $B_{d_\X}(\Mo,2\diam(\Mo))$, which we now prove to be compact. Recall that the distance is $\operatorname{dist}_{d_{\X}}(x,A) = \inf_{a \in A} d_{\X}(x,a)$. Note that, since $\Mo$ is compact by Assumption \ref{a:Mcompact}, then $\Mo$ is bounded and hence $B_{d_\X}(\Mo,2\diam(\Mo))=  \{x\in \X: \operatorname{dist}_{d_\X}(x, \Mo) \leq 2\text{diam}(\Mo)\}$ is bounded too. \newline
        \noindent We now show that  the function $\operatorname{dist}_{d_\X}(\cdot, \Mo)$ is continuous. We consider a general setting, and let $(X,d)$ be a metric space, $A \subseteq X$ be a bounded subset and $g(x)\coloneqq \operatorname{dist}_{d}(x,A) = \inf_{a \in A} d(x,a)$. We claim that $|g(x)-g(y)|\leq d(x,y)$ for all $x,y \in X$. To see this, consider $a, x, y \in X$ and note that  
			\begin{align*}
			d(x,a) \leq d(x,y) + d(y,a) \quad \implies \quad \inf_{a \in A}d(x,a) \leq d(x,y) + \inf_{a \in A} d(y,a).
			\end{align*}
        Switching the roles of $x$ and $y$, leads to the desired inequality. It follows that $g$ is continuous. Moreover, since the function $\operatorname{dist}_{d_\X}(\cdot, \Mo)$ is continuous, the set 
        $$B_{d_\X}(\Mo,2\diam(\Mo)) = \operatorname{dist}_{d_\X}(\cdot,\Mo)^{-1}([0,2\diam(\Mo)])$$ is closed. Hence, we have proven that $B_{d_\X}(\Mo,2\text{diam}(\Mo))$ is closed and bounded, and by the Heine Borel property of $d_\X$ granted by Assumption \ref{a:Mcompact} it follows that $B_{d_\X}(\Mo,2\text{diam}(\Mo))$ is compact. This proves that  $\varphi$ is weakly-measurable with non-empty compact values.
		 Secondly, to prove that $f$ is Carathéodory, we need to show that $f(y, \cdot) = f_y$ is continuous for every fixed $y\in\Mt^{\E}$ and that $f(\cdot,z)$ is measurable for every fixed $z\in\Mo$. On the one hand, for every fixed $y\in\Mt^{\E}$, the function $f_y$ is continuous on $\X$ as proven in claim (I).
			On the other hand, for every fixed $z$, the function $f(\cdot,z): y \mapsto  \int_{F^{-1}(y)} d_\X(x ,z)^p \ d\mu^y(x,e)$ is Borel measurable described by Definition \ref{def:disintegration} (ii).
		Then, by Theorem \ref{thm:mmt}, the possibly set-valued function $\Phi: \Mt^{\E} \rightrightarrows \X$ given by
		\begin{align*}
		\Phi(y) = \argmin_{z \in B_{d_\X}(\Mo,2\text{diam}(\Mo)) } \ \int_{F^{-1}(y)} d_\X(x ,z)^p \ d\mu^y(x,e)
		\end{align*}
		is measurable and has non-empty, compact values. Moreover, by combining \eqref{eq:noiselessargmin23ee} and the fact that for $(x,e) \in G_y$, $B_{d_\X}(x,2r_y) \subseteq B_{d_\X}(\Mo,2\text{diam}(\Mo)) $, we deduce that $\Phi = \Psi$. Hence, $\Psi$ is measurable and has non-empty, compact values.
		
		 \noindent We now prove (c), namely that $\Psi \in \mathcal{C}$. This follows directly from Proposition \ref{prop:C}.
		
		\noindent  We now proceed to prove (d), namely that $\Psi$ is an optimal map.	Let $\varphi \in \mathcal{C}$.  By the minimising definition of $\Psi$,
		\begin{align*}
		\int_{F^{-1}(y)} d_\X^H(\Psi(y),x)^p \ d\mu^y(x,e)  \leq \int_{F^{-1}(y)} d_\X(z,x)^p \ d\mu^y(x,e) 
		\end{align*}
		for every $z \in \varphi(y)$. In particular, taking the supremum with respect to $z \in \varphi(y)$, which coincides with considering the Hausdorff distance, and using Fatou's Lemma yields
		\begin{align*}
		\int_{F^{-1}(y)} d_\X^H(\Psi(y),x)^p \ d\mu^y(x,e) & \leq \sup_{z \in \varphi(y)} \int_{F^{-1}(y)} d_\X(z,x)^p \ d\mu^y(x,e)  \\
		& \leq  \int_{F^{-1}(y)} \sup_{z \in \varphi(y)} d_\X(z,x)^p \ d\mu^y(x,e)  \\
		& \leq  \int_{F^{-1}(y)} d_\X^H(\varphi(y),x)^p \ d\mu^y(x,e).
		\end{align*}
		By integrating with respect to $y\in\Mt^{\E}$, we obtain
		\begin{align*}
		\int_{y \in \Mt^{\E}} &\int_{F^{-1}(y)} d_\X^H(\Psi(y),x)^p \ d\mu^y(x,e) \ d(F_*\mu)(y) \\
        &\leq 
		\int_{y \in \Mt^\E} \int_{F^{-1}(y)} d_\X^H(\varphi(y),x)^p \ d\mu^y(x,e) \ d(F_*\mu)(y).
		\end{align*}
		Due to Definition \ref{def:disintegration} (iii) of the disintegration of measure, the above integrals can be rewritten as
		\begin{align*}
		\int_{\Mo\times\E}  d_\X^H(\Psi(F(x,e)),x)^p \ d\mu(x,e) \leq 
		\int_{\Mo\times\E}  d_\X^H(\varphi(F(x,e)),x)^p \ d\mu(x,e).
		\end{align*}
		Now, as $\varphi \in \mathcal{C}$ was arbitrary and by raising both sides to the power  $\frac{1}{p}$, we obtain
		\begin{align*}
		\Bigg(\int_{(x,e)\in \Mo\times\E}  d_\X^H(\Psi(F(x,e)),x)^p \ d\mu(x,e) \Bigg)^\frac{1}{p} \leq c_{\mathrm{opt}}^{\text{a}}(F,\Mo,  \E,  p).
		\end{align*}
		The opposite inequality holds trivially, as $\Psi \in \mathcal{C}$. Therefore, $\Psi$ is an optimal map.
		
		 \noindent Finally, we proceed to prove (e), namely the upper bound $	c_{\mathrm{opt}}^{\text{a}}(F,\Mo, \E, p) \leq \operatorname{kersize}^{\text{a}}(F,\Mo, \E, p).$ By the minimisation property of $\Phi$, hence also of $\Psi$, for every $(x',e') \in F^{-1}(y)$ and for $F_*\mu$-a.e. $y \in \mathcal{M}_2^{\mathcal{E}}$:
		\begin{align}\label{eq:infimumdel2}
		\int_{F^{-1}(y)} d_\X^H(x ,\Psi(y))^p \ d\mu^y(x,e) \leq \int_{F^{-1}(y)} d_\X(x ,x')^p \ d\mu^y(x,e).
		\end{align}
		Integrating \eqref{eq:infimumdel2} with respect to $\mu^y$ in the $(x',e')$ variables yields,
		\begin{align*}
		\int_{F^{-1}(y)} d_\X^H(x, \Psi(y))^p \ d\mu^y(x,e) \nonumber \leq \int_{F^{-1}(y)}	\int_{F^{-1}(y)} d_\X(x, x')^p \ d\mu^y(x,e) \ d\mu^y(x',e').
		\end{align*}
		where we used that $\mu^y$ is a probability measure.
Integrating both sides over $y\in\Mt^{\E}$ with respect to $F_*\mu$ we obtain
		\begin{align*}
		&\int_{\Mt^{\E}} \int_{F_y} d_\X^H(x, \Psi(y))^p \ d\mu^y(x,e) \ d(F_*\mu)(y)\\
		&\leq 	\int_{\Mt^{\E}} \int_{F^{-1}(y)} \int_{F^{-1}(y)} d_\X(x, x')^p\ d\mu^y(x,e)  \ d\mu^y(x',e') \ d(F_*\mu)(y).
		\end{align*}
		Using the definition of disintegration of the measure $\mu$ on the left hand side of the above inequality yields
		\begin{align*}
		\int_{\Mo\times\E} d_\X^H(x, \Psi(F(x,e)))^p \ d\mu(x,e) \leq\operatorname{kersize}^{\text{a}}(F,\Mo, \E, p)^p.
		\end{align*}
		Then, raising both sides to the $\frac{1}{p}$-th power, gives
		\begin{align*}
		\Bigg(\int_{\Mo\times\E} d_\X^H(x, \Psi(F(x,e)))^p \ d\mu(x,e) \Bigg)^\frac{1}{p} \leq \operatorname{kersize}^{\text{a}}(F,\Mo, \E, p).
		\end{align*}
		And finally, since $\Psi \in \mathcal{C}$, we conclude 
		\begin{align*}
		c_{\mathrm{opt}}^{\text{a}}(F,\Mo, \E, p) \leq \operatorname{kersize}^{\text{a}}(F,\Mo, \E, p).
		\end{align*}	
		This concludes the case $p \in [1,+\infty)$. Hence the proof is complete.

		\emph{Now we consider the case $p=\infty$}. The proof only requires minor modifications to the case $p \in [1,\infty)$, that we will describe in the following. The objective function for $y \in \Mt^{\E}$ and $z \in \X$ is
			\begin{equation*}
				f_y(z) = \essup_{(x,e) \in F^{-1}(y)} d_\X(x, z).
			\end{equation*}
			where the essential supremum is taken with respect to $\mu^y$. Then $\Psi(y) = \argmin_{\X} f_y$. Define, as before,
			\begin{equation*}
				r_y = \essup_{\substack{(x,e) \in F^{-1}(y)\\(x',e') \in F^{-1}(y)}} d_\X(x, x').
			\end{equation*}
			\begin{claim}
				For every $y \in \Mt^\E$: 
				\begin{enumerate}
					\item[(I)] $f_y$ is continuous;
					\item[(II)] $\underset{\X}{\argmin}f_y =\underset{B(x,2r_y)}{\argmin}f_y$ for $\mu^y$-almost every $(x,e)\in F^{-1}(y)$.
				\end{enumerate}
			\end{claim}
		The claims are proven analogously to the case $p \in [1,\infty)$, except that to prove (I), Minkowski's inequality is replaced by the following. We consider a general setting, and let $(X,d)$ be a metric space, $A\subseteq X$ be a bounded subset, $\nu$ be a probability measure on $X$ concentrated on $A$ and $g(x)\coloneqq \essup_{a \in A} d(x,a)$. We claim that $|g(x)-g(z)|\leq d(x,z)$ for all $x,z \in X$. To see this, consider $a, x, z \in X$ and notice that  
			\begin{align*}
				d(x,a) \leq d(x,z) + d(z,a) \quad \implies \quad \essup_{a \in A}d(x,a) \leq d(x,z) + \essup_{a \in A} d(z,a).
			\end{align*}
			Switching the roles of $x$ and $z$, leads to the desired inequality. It follows that $g$ is continuous.
		Claim (II) follows the same line of reasoning as in the case $p \in [1,\infty)$ by replacing the integration with the essential supremum with respect to $\mu^y$.
		
		With the claim, we can show the properties given by the list (a)-(e) stated at the beginning of the proof, similarly to the case $p \in [1,\infty)$. For the sake of brevity, we provide a sketch of the proof.

	Part (a) follows the same line of reasoning as in the case $p \in [1,\infty)$. In part (b) we again apply the Maximum Measurable Theorem \ref{thm:mmt} with $f(y,z) = f_y(z) := \essup_{(x,e) \in F^{-1}(y)} d_\X(x ,z)$ and apply Proposition \ref{prop:disintegration_properties} to show that $f(\cdot,z): y \mapsto  \essup_{(x,e)\in F^{-1}(y)}d_\X(x,z)$ is Borel measurable. Part (c) is again a a direct consequence of (b) and Proposition \ref{prop:C}. Parts (d) and (e) follows the same line of reasoning as in the case $p \in [1,\infty)$ by replacing the integration with the essential supremum. This concludes the proof of the proposition in the case $p = \infty$.
		
\end{proof}

\subsection{Proof of Proposition \ref{prop:rnspkers}}

\begin{proof}
We begin with part $(1)$ by showing that if $A$ satisfies the rNSP, \eqref{eq:rNSP}, with respect to the set $\Mo$, then $A$ is injective on $\Mo$. Let $x,x' \in \mathcal{M}_1$ such that $Ax=Ax'$, then $A(x-x')=0$. Now we apply the rNSP, \eqref{eq:rNSP}, to $h\coloneq x-x'$, and since $x-x' \in \Mo - \Mo$, this yields
\begin{align*}
\vertiii{x-x'}_{\X_1} \leq D_1 \mathrm{dist}_{\X_2}(x-x',\mathcal{M}_1-\mathcal{M}_1) + 0 = 0.
\end{align*}
This implies that $h=x-x'=0$. Hence $A$ is injective on $\Mo$. \newline
Additionally, if $A$ is injective on $\Mo$, then for $x,x' \in \mathcal{M}_1$ such that $Ax=Ax'$, then $x-x'=0$ and we have that
\begin{align*}
    (\Mo - \Mo) \cap \mathcal{N}(A) = \{0\}.
\end{align*}  
For the converse direction, let $x,x' \in \Mo$ such that $Ax=Ax'$. This is equivalent to $A(x-x')=0$ and as $(\Mo - \Mo) \cap \mathcal{N}(A) = \{0\}$, we have that $x-x'=0$. Thus, $x=x'$ and $A$ is injective.\newline

\noindent For part $(2)$, let $y \in \Mt^\E$ and $x,x'\in \mathcal{M}_1$ and $e,e'\in \mathcal{E}$ be such that $y=F(x,e) = Ax+e =Ax'+e'= F(x',e')$. From the rNSP \eqref{eq:rNSP}, as $\mathrm{diam}_{d_{\Y}}(\mathcal{E}) = \sup_{e,e' \in \E}\vertiii{e-e'}_{\Y} = \eta$ and since $x-x' \in \Mo - \Mo$, we then have that 
\begin{align*}
    \vertiii{x-x'}_{\X_1} &\leq D_1\mathrm{dist}_{\X_2}(x-x', \mathcal{M}_1 -\mathcal{M}_1) + D_2 \vertiii{A(x-x')}_{\Y} \\
    &= D_2\vertiii{e-e'}_{\Y} \leq D_2\eta.
\end{align*} 
Taking supremum over all $x, x' \in F_y$ and, then, over $y \in \Mt^\E$ yields
\begin{align*}
\mathrm{kersize}^{\text{w}}(F, \mathcal{M}_1, \mathcal{E})
= \sup_{y \in \mathcal{M}_2^{\mathcal{E}}} \ \sup_{x,x' \in F_y} d_{\X}(x,x')
\le D_2\eta.
\end{align*}
\end{proof}

\subsection{Proof of Proposition \ref{prop:Bayes}}

Before introducing the Bayesian framework for inverse problems, we recall the concept of regular conditional distribution. 

\begin{definition}[{Regular conditional distribution, \cite[$\S$ 10.2]{dudley2018real}}]\label{def:regconddist}
Let $(\Omega, \mathcal{F}, \mathbb{P})$ be a probability space, and let $\mathbf{x} : (\Omega, \mathcal{F}) \to (\X, \mathcal{B}(\X))$ and $\mathbf{y} : (\Omega, \mathcal{F}) \to (\Y, \mathcal{B}(\Y))$ be random variables. The regular conditional distribution of $\mathbf{x}$ given $\mathbf{y}$ is a function $\mathbb{P}_{\mathbf{x} \mid \mathbf{y} = \cdot} = \mathbb{P}_{\mathbf{x}}[\cdot \mid \mathbf{y} = \cdot] : \mathcal{B}(\X) \times \Y \to [0,1]$
such that:
\begin{enumerate}
    \item For every $y \in \Y$, the function $\mathbb{P}_{\mathbf{x} \mid \mathbf{y} = y} : \mathcal{B}(\X) \to [0,1]$ is a probability measure; 
    \item For every $A \in \mathcal{B}(\X)$, the function $\mathbb{P}_{\mathbf{x} \mid \mathbf{y} = \cdot}(A) : \Y \to [0,1]$ is $(\mathcal{B}(\Y), \mathcal{B}([0,1]))$-measurable and
    \begin{equation}
    \mathbb{P}_{\mathbf{x}}(A)= \int_{\Y} \mathbb{P}_{\mathbf{x} \mid \mathbf{y} = y}(A)\, d\mathbb{P}_{\mathbf{y}}(y),
    \tag{3.5.1}
    \end{equation}
    where $\mathbb{P}_{\mathbf{x}} = \mathbb{P} \circ \mathbf{x}^{-1}$ and $\mathbb{P}_{\mathbf{y}} = \mathbb{P} \circ \mathbf{y}^{-1}$ are the laws of $x$ on $(\X, \mathcal{B}(\X))$ and of $y$ on $(\Y, \mathcal{B}(\Y))$, respectively.
\end{enumerate}
\end{definition}

\noindent Regular conditional distributions do not always exist. The following positive result guarantees that a regular conditional distribution exists when the random variable that is being conditioned takes values in a completely metrizable and separable topological space, also referred to as a Polish space.

\begin{theorem}[{Existence and uniqueness of regular conditional distributions \cite[Theorem 10.2.2]{dudley2018real} or \cite[Theorem B.32 and Lemma B.40]{schervish2012theory}}]
Let $(\Omega, \mathcal{F}, \mathbb{P})$ be a probability space and let $\Y$ be a Polish space. Let $\mathbf{x} : (\Omega, \mathcal{F}) \to (\X, \mathcal{B}(\X))$ and $\mathbf{y} : (\Omega, \mathcal{F}) \to (\Y, \mathcal{B}(\Y))$ be random variables. Then, there exists a regular conditional distribution $\mathbb{P}_{\mathbf{x} \mid \mathbf{y} = \cdot} = \mathbb{P}_{\mathbf{x}}[ \cdot \mid \mathbf{y} = \cdot] : \mathcal{B}(\X) \times \Y \to [0,1]$ of $\mathbf{x}$ given $\mathbf{y}$. Moreover, the regular conditional distribution is unique up to $\mathbb{P}$-almost sure equivalence.
\end{theorem}

\begin{remark}[Disintegrations of measure are regular conditional probabilities]
Let $(\Omega, \mathcal{F}, \mathbb{P})$ be a probability space and let $\X$ be a Polish space. Let $\mathbf{x} : (\Omega, \mathcal{F}) \to (\X, \mathcal{B}(\X))$ be a random variable. Equip $(\X, \mathcal{B}(\X))$ with a probability measure given by the law of $\mathbf{x}$, as $\mu_{\mathbf{x}} = \mathbb{P}_{\mathbf{x}}(\cdot)$. Let $F: (\X, \mathcal{B}(\X)) \rightarrow (\Y, \mathcal{B}(\Y))$ be measurable. Define the random variable $\mathbf{y} := F \circ \mathbf{x}$ on $\Y$. Then, a disintegration $\{\mu_{\mathbf{x}}^y\}_{y\in \Y}$ of $\mu_{\mathbf{x}}$ along $F$ defines a regular conditional distribution of $\mathbf{x}$ given $\mathbf{y}$. The reverse, that a regular conditional distribution is a disintegration of measure, in general does not hold; as in Definition \ref{def:disintegration} (i) of the disintegration of measure, the support is required to concentrate on the preimages $F^{-1}(y)$. 
\end{remark}

\begin{proof}
We start with part $(1)$. Starting from the joint random variable $(\mathbf{x},\mathbf{e}) \sim \mu_{\mathbf{x},\mathbf{e}} = \mathbb{P}[(\mathbf{x},\mathbf{e}) \in \cdot]$, the random variable $\mathbf{x}$ can be obtained by the composition $\mathbf{x} = \pi_1 \circ (\mathbf{x},\mathbf{e})$. Then, its distribution is the Bayesian prior of the inverse problem, and is given by the marginal distribution $\mu_{\mathbf{x}} = \pi_{1*} \mu_{\mathbf{x},\mathbf{e}} = \mathbb{P}[\mathbf{x} \in \cdot]$.

For part $(2)$ apply proposition \ref{prop:disint}, but consider a different measurable function and different spaces. We disintegrate $\mu$ along the projection $\pi_1 : (\X \times \E, \mathcal{B}(\X \times \E)) \to (\X, \mathcal{B}(\X))$. By Proposition \ref{prop:disint} there exists a disintegration $\{\mu^x\}_{x \in \X}$ of $\mu$ along $\pi_1$ and each measure $\mu^x$ is a probability measure that is concentrated on $\{x\} \times \mathcal{E}$. Then, this disintegration $\mu^x$ can be pushed forward via $F$, as $\mu_{\mathbf{y} \mid \mathbf{x} = x} := F_* \mu^x$. Now we show that  $\mu_{\mathbf{y} \mid \mathbf{x} = \cdot}$ is a regular conditional distribution of $\mathbf{y}$ given $\mathbf{x}$. As $\mu^x$ is a probability measure, for $A \in \mathcal{B}(\Mt^\E)$ we have that $\mu_{\mathbf{y} \mid \mathbf{x} = x}(A) = \mu^x(F^{-1}(A)) \in [0,1]$. Additionally, we have that $\mu_{\mathbf{y} \mid \mathbf{x} = x}(\Mt^\E) = \mu^x(F^{-1}(\Mt^\E))=1$, since $\mu^x$ is a probability measure and as $\mu$ is supported on $\Mo \times \E$ by Assumption \ref{a:measur}. Thus $\mu_{\mathbf{y} \mid \mathbf{x} = x}$ is a probability measure. For all $V \in \mathcal{B}(\X\times\E)$ the function $x \mapsto \mu^x(V)$ is measurable. For $A \in \mathcal{B}(\Mt^\E)$ as $F$ is measurable, $V = F^{-1}(A)$ is measurable and we have that $x \mapsto \mu_{\mathbf{y} \mid \mathbf{x} = x}(A) = \mu^x(F^{-1}(A))$ is measurable. Now recall that by Definition \ref{def:disintegration} (iii), we have that the disintegration of measure satisfies
\begin{align*}
\mu(A) = \int_{\X \times \E} 1_A(x,e) d\mu(x,e) = \int_{\X} \left(\int_{\E} 1_A(x,e) d\mu^x(x,e)\right) d\pi_{1*}\mu(x) = \int_{\X} \mu^x(A) d\pi_{1*}\mu(x).
\end{align*}
Now by Definition \ref{def:disintegration} (iii) and as $\mu_{\mathbf{y}} = F_* \mu_{\mathbf{x},\mathbf{e}}$, the property $(2)$ in Definition \ref{def:regconddist} of conditional distributions is verified:
\begin{align*}
\mu_{\mathbf{y}}(A) &= F_* \mu(A) = \mu_{\mathbf{x},\mathbf{e}}(F^{-1}(A))\\
&= \int_\X \mu^x(F^{-1}(A))\, d\pi_{1*}\mu(x)
= \int_\X F_* \mu^x(A)\, d\pi_{1*}\mu(x)\\
&= \int_\X \mu_{\mathbf{y} \mid \mathbf{x} = x}(A)\, d\mu_{\mathbf{x}}(x),
\end{align*}
for every $A \in \mathcal{B}(\Mt^\E)$. This construction proves that $\{\mu_{\mathbf{y} \mid \mathbf{x} = x}\}_{x \in \mathcal{X}}$ is a likelihood distribution in the Bayesian sense.
For part $(3)$ we apply Proposition \ref{prop:disint} to disintegrate $\mu$ along the function $F : (\X \times \E, \mathcal{B}(\X \times \E)) \to (\Mt^\E, \mathcal{B}(\Mt^\E))$ and, then, pushforward the resulting disintegration $\{\mu^y\}_{y \in \Mt^\E}$ with respect to $\pi_1$ to obtain the Bayes posterior $\{\pi_{1*}\mu^y\}_{y \in \Mt^\E}$. The Bayes posterior is a family of regular conditional distributions $\{\mu_{\mathbf{x} \mid \mathbf{y} = y}\}_{y \in \Mt^\E}$ that is uniquely defined (up to $\mu_{\mathbf{y}}$-almost sure equivalence) by satisfying the condition
\begin{align}\label{eq:regdiscond}
\mu_{\mathbf{x}}(A) = \int_{\Mt^\E} \mu_{\mathbf{x} \mid \mathbf{y} = y}(A)\, d\mu_{\mathbf{y}}(y),
\end{align}
for all $A \in \mathcal{B}(\X)$. By Proposition \ref{prop:disint} there exists a disintegration $\{\mu^y\}_{y \in \Mt^\E}$ of $\mu$ along $F$ and each measure $\mu^y$ is a probability measure that is concentrated on $F^{-1}(y) \subset \X \times \E$. We now prove that the family $\{\pi_{1*} \mu^y\}_{y \in \Mt^\E}$ satisfies the condition \eqref{eq:regdiscond} of a regular conditional distribution. For $A \in \mathcal{B}(\X)$, by Definition \ref{def:disintegration} (iii) we have
\begin{align*}
\mu_{\mathbf{x}}(A) = (\pi_{1*} \mu)(A) = \mu(\pi_1^{-1}(A))
    = \int_{\Mt^\E} \mu^y(\pi_1^{-1}(A))\, dF_*\mu(y)
    = \int_{\Mt^\E} \pi_{1*}\mu^y(A)\, d\mu_{\mathbf{y}}(y).
\end{align*}
Therefore, the defining condition for the posterior is verified. Since $(\Y,d_\Y)$ is a Polish space, the regular conditional distribution 
$\mu_{\mathbf{x} \mid \mathbf{y} = \cdot}$ exists and is unique up to 
$\mu_{\mathbf{y}}$-almost sure equivalence by \cite[Theorem 10.2.2]{dudley2018real}. We conclude that $\pi_{1*} \mu^y = \mu_{\mathbf{x} \mid \mathbf{y} = y}$ for $\mu_{\mathbf{y}}$-a.e. $y \in \Mt^\E$.

Now for part $(4)$ we apply Theorem \ref{thm:optimal_bounds} part (ii). As the integrand does not depend on $e$, the optimal map can be rewritten depending on the posterior distribution from part $(3)$
\begin{align}
\Psi(y) &= \argmin_{z \in \X} \int_{F^{-1}(y)} d_\X(x,z)^p \ d\mu^y(x,e), \\
&= \argmin_{z \in \X} \int_{\pi_1(F^{-1}(y))} d_\X(x,z)^p \ d(\pi_{1*}\mu^y)(x)
\end{align}
Recall that $\pi_{1*}\mu^y = \mu_{\mathbf{x} \mid \mathbf{y}=y}$. As $d_\X$ is induced by an inner product $\langle \cdot, \cdot \rangle$ and the norm $\|\cdot\|$ induced by the inner product, for $p = 2$, and the measure $\pi_{1*}\mu^y$ is supported on $\pi_1(F^{-1}(y))$, the optimal map reduces to
\begin{align*}
\Psi(y)
=
\arg\min_{z \in \X} \int_{\X} \|x - z\|^2\, d\mu_{\mathbf{x} \mid \mathbf{y}=y}(x)
=
\arg\min_{z \in \X} \mathbb{E}_{\mu_{\mathbf{x} \mid \mathbf{y}=y}}\big[\|\mathbf{x} - z\|^2\big],
\end{align*}
which is the Minimum Mean Squared Error (MMSE) for the posterior distribution $\pi_{1*}\mu^y = \mu_{\mathbf{x} \mid \mathbf{y}=y}$. Using bilinearity of the inner product and linearity of expectation, the objective function to be minimised in the expression for the optimal map $\Psi$ can be rewritten as:
\begin{align*}
\mathbb{E}_{\mu_{\mathbf{x} \mid \mathbf{y}=y}}\big[\|\mathbf{x} - z\|^2\big]
&= \mathbb{E}_{\mu_{\mathbf{x} \mid \mathbf{y}=y}}\big[\|\mathbf{x} - \mathbb{E}_{\mu_{\mathbf{x} \mid \mathbf{y}=y}}[\mathbf{x}]  + \mathbb{E}_{\mu_{\mathbf{x} \mid \mathbf{y}=y}}[\mathbf{x}] - z\|^2\big] \\
&= \mathbb{E}_{\mu_{\mathbf{x} \mid \mathbf{y}=y}}\big[\|\mathbf{x} - \mathbb{E}_{\mu_{\mathbf{x} \mid \mathbf{y}=y}}[\mathbf{x}] \|^2\big] + \mathbb{E}_{\mu_{\mathbf{x} \mid \mathbf{y}=y}}\big[\|\mathbb{E}_{\mu_{\mathbf{x} \mid \mathbf{y}=y}}[\mathbf{x}] - z\|^2\big] \\ 
& + 2\mathbb{E}_{\mu_{\mathbf{x} \mid \mathbf{y}=y}}\big\langle \mathbf{x} - \mathbb{E}_{\mu_{\mathbf{x} \mid \mathbf{y}=y}}[\mathbf{x}], \mathbb{E}_{\mu_{\mathbf{x} \mid \mathbf{y}=y}}[\mathbf{x}] - z\big\rangle\\
&= \mathbb{E}_{\mu_{\mathbf{x} \mid \mathbf{y}=y}}\big[\|\mathbf{x} - \mathbb{E}_{\mu_{\mathbf{x} \mid \mathbf{y}=y}}[\mathbf{x}] \|^2\big] + \mathbb{E}_{\mu_{\mathbf{x} \mid \mathbf{y}=y}}\big[\|\mathbb{E}_{\mu_{\mathbf{x} \mid \mathbf{y}=y}}[\mathbf{x}] - z\|^2\big] \\ 
& + 2\big\langle \mathbb{E}_{\mu_{\mathbf{x} \mid \mathbf{y}=y}}[\mathbf{x}] - \mathbb{E}_{\mu_{\mathbf{x} \mid \mathbf{y}=y}}[\mathbf{x}], \mathbb{E}_{\mu_{\mathbf{x} \mid \mathbf{y}=y}}[\mathbf{x}] - z\big\rangle\\
&= \mathbb{E}_{\mu_{\mathbf{x} \mid \mathbf{y}=y}}\big[\|\mathbf{x} - \mathbb{E}_{\mu_{\mathbf{x} \mid \mathbf{y}=y}}[\mathbf{x}] \|^2\big] + \mathbb{E}_{\mu_{\mathbf{x} \mid \mathbf{y}=y}}\big[\|\mathbb{E}_{\mu_{\mathbf{x} \mid \mathbf{y}=y}}[\mathbf{x}] - z\|^2\big] 
\end{align*}
which is the sum of two non-negative quantities, and is thus minimised when
$z = \mathbb{E}_{\mu_{\mathbf{x} \mid \mathbf{y}=y}}[\mathbf{x}]$. Thus, the optimal map is given by $\Psi(y) = \mathbb{E}_{\mu_{\mathbf{x} \mid \mathbf{y}=y}}[\mathbf{x}]$ and coincides with the posterior mean.
\end{proof}

\subsection{Funding statement}
The work of the first author was supported by the UK EPSRC Centre for Doctoral
Training. The work of the last author was supported by Royal Society University Research Fellowship and the Leverhulme Prize 2017. 

\subsection{Competing interest} The authors declare no competing interests.

\bibliographystyle{abbrv}
\bibliography{references}

\end{document}